\documentclass[letterpaper,11pt,reqno]{article}

\usepackage[margin=3cm]{geometry}% big spacing for "draft mode"
\usepackage{setspace}
\usepackage{ragged2e}

\usepackage{amsmath}
\usepackage{amsthm}
\usepackage{mathtools,amsfonts}
\usepackage{amssymb,mathrsfs,xspace,multicol,stmaryrd}
\usepackage[mathscr]{eucal}
\usepackage{amscd}

\usepackage[all]{xy}
\usepackage[usenames,dvipsnames]{xcolor}
\usepackage[textsize=scriptsize,textwidth=2cm]{todonotes}

\usepackage{currfile}
\usepackage[ddmmyyyy]{datetime}

\usepackage[breaklinks=true,colorlinks=true,citecolor=blue,urlcolor=MidnightBlue,linkcolor=black]{hyperref}

\usepackage[shortlabels]{enumitem}
\usepackage{xspace}
\usepackage{tikz}
\usepackage{tikz-cd}
\usepgflibrary{arrows}
\usetikzlibrary{matrix}
\usetikzlibrary{calc}
\usetikzlibrary{positioning}

\tikzset{%
  operad style/.style={fill=white,inner sep=2.5pt},
  baseline=(current  bounding  box.center),
  ampersand replacement=\&, row sep=2em,column sep=2em,
  std/.style={->,font=\scriptsize}
}

\tikzset{%
  clr style/.style={fill=white,inner sep=2.5pt},
  baseline=(current  bounding  box.center),
  ampersand replacement=\&, row sep=2em,column sep=2em,
  std/.style={->,font=\scriptsize}
}

\newenvironment{tikzdiag}[2]
{
\begin{tikzpicture}[clr style]
\matrix (m) [matrix of math nodes, row sep=#1em, column sep=#2em]
}
{
\end{tikzpicture}
}

\def\release#1{\let#1\undefined}

\newlist{pfitems}{itemize}{1}
\setlist[pfitems]{label=--}

\makeatletter
\providecommand{\leftsquigarrow}{%
  \mathrel{\mathpalette\reflect@squig\relax}%
}
\newcommand{\reflect@squig}[2]{%
  \reflectbox{$\m@th#1\rightsquigarrow$}%
}
\makeatother

\newcommand{\Z}{\mathbb{Z}}
\newcommand{\kk}{\Bbbk}
\newcommand{\F}{\mathbb{F}}
\newcommand{\FF}{\mathbb{F}}
\newcommand{\mm}{\mathfrak{m}}
\newcommand{\curv}{{\mathsf{curv}}}

\DeclareMathOperator{\MC}{\mathrm{MC}}
\newcommand{\MCAs}{\MC_{\Ass}}
\DeclareMathOperator{\sMC}{\mathcal{MC}}

\DeclareMathOperator{\Hom}{\mathrm{Hom}}

\DeclareMathOperator{\End}{\mathrm{End}}
\DeclareMathOperator{\id}{\mathrm{id}}

\newcommand{\op}{\mathrm{op}}

\newcommand{\cat}[1]{\mathsf{#1}}
\newcommand{\catC}{\cat{C}}

\newcommand{\dga}{\mathsf{dga}}

\newcommand{\Ch}{\mathsf{Ch}}
\newcommand{\Chain}{\mathsf{Ch}_{\geq 0}}
\newcommand{\cChain}{\mathsf{Ch}^{\geq 0}}

\newcommand{\Set}{\mathsf{Set}}
\newcommand{\sSet}{s\mathsf{Set}}
\newcommand{\Kan}{\mathsf{KanCplx}}

\DeclareMathOperator{\cdga}{\mathsf{ComdgAlg}}
\DeclareMathOperator{\Vect}{\mathsf{Vect}}
\newcommand{\maps}{\colon}

\newcommand{\xto}[1]{\xrightarrow{#1}}

\newcommand{\fib}{\twoheadrightarrow}
\newcommand{\weq}{\xrightarrow{\sim}}
\newcommand{\trivfib}{\overset{\sim}{\twoheadrightarrow}}
\newcommand{\afib}{\trivfib}

\newcommand{\emb}{\hookrightarrow}

\newcommand{\al}{\alpha}
\newcommand{\be}{\beta}

\newcommand{\ga}{\gamma}
\newcommand{\Del}{\Delta}
\newcommand{\del}{\delta}

\newcommand{\ro}{\rho}
\newcommand{\Om}{\Omega}
\newcommand{\om}{\omega}
\newcommand{\tha}{\theta}
\newcommand{\Ph}{\Phi}

\newcommand{\vphi}{\varphi}
\newcommand{\vph}{\varphi}

\newcommand{\el}{\ell}
\newcommand{\si}{\sigma}

\newcommand{\ta}{\tau}

\newcommand{\bb}[1]{\vec{#1}}
\newcommand{\ba}[1]{\bar{#1}}
\newcommand{\ti}[1]{\tilde{#1}}

\newcommand{\wh}[1]{\widehat{#1}}

\newcommand{\und}[1]{{\underline{#1}}}
\newcommand{\ov}[1]{{\overline{#1}}}
\newcommand{\bul}{\bullet}

\newcommand{\cF}{\mathcal{F}}

\newcommand{\cB}{\mathcal{B}}

\newcommand{\cL}{\mathcal{L}}
\newcommand{\cc}{\circ}
\newcommand{\iso}{\cong}

\newcommand{\sse}{\subseteq}

\newcommand{\st}{~ \left. \right \vert ~} %"such that" or "|" for set builder notation

\DeclareMathOperator{\diag}{\mathrm{diag}}

\newcommand{\spann}{\mathrm{span}}
\newcommand{\tensor}{\otimes}

\newcommand{\dsum}{\oplus}

\newcommand{\ev}{\mathrm{ev}}
\newcommand{\pr}{\mathrm{pr}}

\DeclareMathOperator{\ppr}{Pr}
\DeclareMathOperator{\im}{\mathrm{im}}

\DeclareMathOperator{\sym}{\mathrm{sym}}

\renewcommand{\deg}[1]{\left \lvert #1 \right \rvert}

\newcommand{\T}{\bar{T}}

\newcommand{\sQ}{Q^{1}}

\newcommand{\bs}{\mathbf{s}}

\newcommand{\simh}{\simeq_{\Delta}}

\newcommand{\plim}{\varprojlim}

\newcommand{\cinf}{C^{\infty}}

\newlist{pfcases}{itemize}{1}
\setlist[pfcases]{label=--}
\newlist{pfsteps}{itemize}{1}
\setlist[pfsteps]{label={}}

\theoremstyle{plain}
\newtheorem{theorem}{Theorem}[section]
\newtheorem*{theorem*}{Theorem}
\newtheorem*{theorem1}{Theorem 1}
\newtheorem*{theorem2}{Theorem 2}
\newtheorem*{theorem3}{Theorem 3}

\newtheorem{proposition}[theorem]{Proposition}
\newtheorem{lemma}[theorem]{Lemma}
\newtheorem*{lemma*}{Lemma}
\newtheorem*{corollary*}{Corollary}

\newtheorem{corollary}[theorem]{Corollary}

\theoremstyle{definition}
\newtheorem{definition}[theorem]{Definition}

\newtheorem*{ass*}{Assumption}

\newtheorem*{notation*}{Notation}
\newtheorem*{GMThm*}{Goldman-Millson Theorem}
\newtheorem*{GMThmLinf*}{$L_\infty$-Goldman-Millson Theorem}

\theoremstyle{remark}
\newtheorem{remark}[theorem]{Remark}
\newtheorem{example}[theorem]{Example}

\numberwithin{equation}{section}

\newcommand{\qf}[2]{#1/\cF_{#2}#1}

\newcommand{\ctensor}{\widehat{\otimes}}

\renewcommand{\Ch}{\mathsf{Ch}^{\ast}}

\newcommand{\LdQ}[1]{(L^{#1},d^{#1},Q^{#1})}

\newcommand{\spQ}{Q^{\prime 1}}

\renewcommand{\catC}{\cat{C}}

\newcommand{\sThe}{\Theta^1}
\newcommand{\sPsi}{\Psi^1}
\newcommand{\sPh}{\Ph^{1}_1}

\DeclareMathOperator{\cVect}{\widehat{\mathsf{Vect}}}

\DeclareMathOperator{\cCh}{\widehat{\mathsf{Ch}^\ast}}
\DeclareMathOperator{\fVect}{\mathsf{FiltVect}}

\DeclareMathOperator{\cLie}{\widehat{\mathsf{Lie}}\lbrack 1\rbrack_\infty}

\DeclareMathOperator{\strcLie}{\widehat{\mathsf{Lie}}\lbrack 1\rbrack^{\mathrm{str}}_\infty}

\newcommand{\ainf}{A_{\infty}}

\newcommand{\linf}{L_{\infty}}
\usepackage[textsize=scriptsize,textwidth=2cm]{todonotes}
\newcommand{\Ass}{\mathrm{As}}
\DeclareMathOperator{\As}{\mathsf{As}\lbrack 1\rbrack_\infty }
\DeclareMathOperator{\cAs}{\widehat{\mathsf{As}}\lbrack 1\rbrack_\infty} 

\newcommand{\AdQ}[1]{(A^{#1},d^{#1},Q^{#1})}
\newcommand{\ph}{\phi}

\newcommand{\NI}{N(\Del^1)}
\newcommand{\tim}{\times}
\newcommand{\Tha}{\Theta}

\newcommand{\ssSet}{ss\mathsf{Set}}
\renewcommand{\dga}{\mathsf{dgAlg}}
\newcommand{\sdga}{s\mathsf{dgAlg}}

\newcommand{\sC}{s\cat{C}}
\newcommand{\ssC}{ss\cat{C}}

\newcommand{\bl}{\bul}

\DeclareMathOperator{\Disc}{Disc}

\newcommand{\VectF}{\Vect(\FF)}
\newcommand{\sVectF}{s\mathsf{Vect}(\FF)}
\newcommand{\fVectF}{\fVect(\FF)}
\newcommand{\cVectF}{\cVect(\FF)}
\newcommand{\ChF}{\Ch(\FF)}
\newcommand{\ChainF}{\Chain(\FF)}
\newcommand{\cChF}{\cCh(\FF)}
\newcommand{\dgaF}{\mathsf{dgAlg}(\FF)}
\newcommand{\ssh}{\mathrm{sh}}
\newcommand{\shpr}[2]{#1 \star_{\ssh} #2}

\usepackage[makeroom]{cancel}

\renewcommand\footnotemark{}
\renewcommand\footnotemark{}

\newcommand{\cupp}{\smallsmile}
\newcommand{\J}{\mathfrak{I}}
\newcommand{\Tc}{\bar{T}^c}

\newcommand{\sainf}{A[1]_{\infty}}
\newcommand{\csainf}{\widehat{A}[1]_{\infty}}
\newcommand{\cainf}{\csainf}
\DeclareMathOperator{\strcAs}{\widehat{\mathsf{As}}\lbrack 1\rbrack^{\mathrm{str}}_\infty}
\newcommand{\ctan}{\tan}
\newcommand{\sN}{\mathcal{N}}
\newcommand{\sNb}{\mathcal{N}_\bl}
\newcommand{\sMCb}{\mathcal{MC}_\bl}

\newcommand{\qm}{\circledast}

\newcommand{\sph}{\mathrm{sph}}
\newcommand{\sphZ}{Z^{\sph}}
\newcommand{\Ord}{\boldsymbol{\Delta}}
\newcommand{\qinv}[1]{\check{#1}}

\newcommand{\unit}{\boldsymbol{1}}
\newcommand{\mult}{\boldsymbol{\cdotp}}

\renewcommand{\S}{Sec.\,}
\newcommand{\df}[1]{\textbf{\textit{#1}}}
\newcommand{\sOmb}{\Om^\ast(\Del^{\bl})}
\newcommand{\sOm}[1]{\Om^\ast(\Del^{#1})}
\newcommand{\Lie}{\mathrm{Lie}}
\newcommand{\sinf}{L[1]_{\infty}}
\renewcommand{\cinf}{\widehat{L}[1]_{\infty}}
\newcommand{\sMCL}[1]{\sMC_{\Lie \, #1}}
\newcommand{\sMCLb}{\sMCL{\bl}}
\newcommand{\sMCA}[1]{\sMC_{\Ass \, #1}}
\newcommand{\sMCAb}{\sMCA{\bl}}

\title{\bf{On the Goldman-Millson theorem for
$A_\infty$-algebras in arbitrary characteristic}}

\author{Alex Milham and Christopher L.\  Rogers}

\date{}

\begin{document}

\maketitle

% %%%BEGIN DRAFT VERSION
% \begin{center}
% {\large \textcolor{red}{DRAFT  version: \currfilename}\\
% {\textcolor{red}{date: \today}}}
% \end{center}
% %%%END DRAFT VERSION

\begin{abstract}
Complete filtered $A_\infty$-algebras model certain deformation problems in the noncommutative setting. The formal deformation theory of a group representation is a classical example. With such applications in mind, we provide the $A_\infty$ analogs of several key theorems from the Maurer-Cartan theory for $L_\infty$-algebras. In contrast with the $L_\infty$ case, our results hold over a field of arbitrary characteristic. We first leverage some abstract homotopical algebra to give a concise proof of the $A_\infty$-Goldman-Millson theorem: The nerve functor, which assigns a simplicial set $\sNb(A)$ to an $A_\infty$-algebra $A$, sends filtered quasi-isomorphisms to homotopy equivalences. We then characterize the homotopy groups of $\sNb(A)$ in terms of the cohomology algebra $H(A)$, and its group of quasi-invertible elements. Finally, we return to the characteristic zero case and
show that the nerve of $A$ is homotopy equivalent to the simplicial Maurer-Cartan set of its commutator $L_\infty$-algebra. This answers a question posed by N.\ de Kleijn and F.\ Wierstra in \cite{DW}. 
\end{abstract}

\setcounter{tocdepth}{1}
\tableofcontents

\newpage
\newcommand{\bel}{\ba{\el}}
\newcommand{\bm}{\ba{m}}
\section{Introduction} \label{sec:intro} 
Just as a commutative algebra has a spectrum, one can functorially assign to a commutative differential graded (dg) algebra its ``simplicial spectrum''. This construction is a key ingredient in D.\ Sullivan's approach \cite{Sullivan} to rational homotopy theory. It also provides a conduit between the (non-abelian) homological algebra of commutative dg algebras, and the homotopy theory of simplicial sets \cite{BG:1976}. In the same spirit, this paper is about constructing simplicial sets from pronilpotent noncommutative dg associative algebras, or more generally, $A_\infty$-algebras.
In order to establish the context and motivation for our results, we first need to recall a well-known story about $L_\infty$-algebras and deformation theory.
 
\paragraph{Complete $L_\infty$-algebras} $L_\infty$-algebras are the central objects in the Maurer-Cartan approach to formal deformation problems over a field of characteristic 0. See, for example, \cite{Man}, \cite{Pridham}, and the expository work \cite{Markl-book} for the precise details behind this statement. Such algebras are homotopical analogues of differential graded (dg) Lie algebras. To be more precise, an $L_\infty$-algebra $L=(L,d,\bel)$ is a cochain complex $(L,d)$ equipped with
an infinite sequence $\bel:=\{\el_2, \el_3,  \cdots\}$ of graded skew-symmetric ``brackets'' $\el_{k} \maps \Lambda^k L \to L$ which satisfy an infinite number of Jacobi-like identities. One recovers the definition of a dg Lie algebra $(L,d,[\cdot,\cdot]:=\el_2)$ if $\el_k=0$ for all $k \geq 3$. 

Within the context of formal deformation theory, the $L_\infty$-algebras involved are {\it complete} i.e.,\ they come equipped with a complete descending filtration $L = \cF_1L \supseteq \cF_2L \supseteq \cdots$ which is compatible with the differential and brackets in the obvious way. The completeness allows one to define the {\it Maurer-Cartan set} $\MC_\Lie(L)$, which consists of all degree 1 elements $x \in L^1$ satisfying the equation 
\[
dx + \sum_{k \geq 2}^\infty (-1)^k\frac{1}{k!}\, \el_k(x^{\wedge k})=0. 
\]
Saying that $L$ models certain 
types of deformations of a mathematical object means that $\MC_\Lie(L)$ is 
in bijection with the set of such deformations. The ``coarse moduli space'' of the deformation problem then corresponds to a quotient of $\MC_\Lie(L)$ by an equivalence relation. This quotient is equal to the set of path connected components $\pi_0(\sMCLb(L))$ of the  {\it simplicial Maurer-Cartan set} $\sMCLb(L):=\MC_\Lie(L \ctensor \sOmb)$. Here,
$L \ctensor \sOmb$ is the simplicial $L_\infty$-algebra obtained by taking the completed tensor product of $L$ with the simplicial dg algebra $\sOmb$ of polynomial de Rham forms. This construction relies crucially on the fact that the multiplication in $\sOmb$ is graded commutative.

When $L$ is a complete dg Lie algebra, the path connected components of $\sMCLb(L)$ coincide with the ``gauge equivalence classes'' of Maurer-Cartan elements. Indeed, $\sMCLb(L)$ can be understood as a large generalization of the {\it  Deligne groupoid} \cite[Sec.\ 2]{GM} associated to a nilpotent dg Lie algebra.
In \cite{GM}, W.\ Goldman and J.\ Millson proved the following theorem:
\begin{GMThm*}[Thm.\ 2.4; \cite{GM}]
 If $f \maps (L,d,[\cdot,\cdot]) \to (L',d',[\cdot,\cdot]')$ is a morphism between
 nilpotent dg Lie algebras concentrated in non-negative degrees, and the underlying morphism $f \maps (L,d) \to (L',d')$ of cochain complexes is a quasi-isomorphism, then $f$ induces a categorical equivalence between their corresponding Deligne groupoids. 
\end{GMThm*}

The Goldman-Millson Theorem is a foundational result in the Maurer-Cartan approach to formal deformation theory.
Building on the work of E.\ Getzler \cite{Getzler}, and V.\ Hinich \cite{Hinich:1997}, a generalization of this theorem to complete $L_\infty$-algebras was given by V.\ Dolgushev and the second author in \cite{DR}. To explain it, we first recall  that an  {\it $\infty$-morphism} $\ba{f} \maps (L,d,\bel) \to (L',d',\bel')$ between complete $L_\infty$-algebras is
an infinite sequence $\ba{f}:=\{f_1, f_2,  \cdots\}$ of graded filtration-preserving skew-symmetric linear maps $f_{k} \maps \Lambda^k L \to L'$ which satisfy an infinite number of identities that  encode the compatibility of the $f_k$ with the brackets. In particular, the identities imply that $f_1 \maps (L,d) \to (L',d')$ is a morphism of filtered cochain complexes. If the restriction $f_1 \vert_{\cF_nL}$ of $f_1$ to each piece of the filtration $(\cF_{n}L,d\vert_{\cF_nL})$ is a quasi-isomorphism of cochain complexes, then $f$ is called a {\it weak equivalence}.

\begin{GMThmLinf*}[Thm.\ 1.1; \cite{DR}] If $\ba{f} \maps 
(L,d,\bel) \to (L',d',\bel')$ is a weak equivalence between complete $L_\infty$-algebras, then $\ba{f}$ induces a homotopy equivalence \\$\sMCLb(\ba{f}) \maps \sMCLb(L) \weq \sMCLb(L')$
between their corresponding simplicial Maurer-Cartan sets. 
\end{GMThmLinf*}

\paragraph{Complete $A_\infty$-algebras}
We now turn to the ``associative'' analogue of the above story, which is where the results in this paper will reside. $A_\infty$-algebras were invented by J.\ Stasheff \cite{Stasheff}; the exposition \cite{Keller} by B.\ Keller is a standard introductory reference. Here it suffices for us to 
recall that an $A_\infty$-algebra $A:=(A,d,\bm)$ is a cochain complex $(A,d)$ equipped an infinite sequence $\bm:=\{m_2, m_3,  \cdots\}$ of graded ``products'' $m_{k} \maps A^{\tensor k} \to A$ which satisfy an infinite number of associative law-type identities. 
There are no symmetry constraints on the $m_k$'s.
One recovers the definition of a non-unital dg associative algebra $(A,d,\mu:=m_2)$ if $m_k=0$ for all $k \geq 3$. Again, we say $A$ is {\it complete} if the graded vector space $A$ is equipped with a complete descending filtration $A = \cF_1A \supseteq \cF_2A \supseteq \cdots$ which is compatible with the differential and products. Just as in the $L_\infty$-case, an  {\it $\infty$-morphism} $\ba{f} \maps (A,d,\bm) \to (A',d',\bm')$ between complete $A_\infty$-algebras is
an infinite sequence $\ba{f}:=\{f_1, f_2,  \cdots\}$ of graded filtration-preserving multilinear maps $f_{k} \maps A^{\tensor k} \to A'$ which satisfy an infinite number of identities. Note that there are no symmetry constraints on the maps $f_k$.

The {\it Maurer-Cartan set} $\MC_\Ass(A)$ consists of all degree 1 elements $a \in A^1$ satisfying the equation 
\[
da + \sum_{k \geq 2}^\infty (-1)^k m_k(a^{\tensor k})=0. 
\]
Note that the Maurer-Cartan equation for a complete $A_\infty$-algebra is well defined over a field $\FF$ of arbitrary characteristic, in contrast to the $L_\infty$-case. Furthermore, 
since there are no symmetry constraints on the products $m_k$, we can tensor $A$ with a noncommutative dg algebra -- or, more generally, a simplicial noncommutative dg algebra --  and obtain a new complete $A_\infty$-algebra. 

The Maurer-Cartan theory of $A_\infty$-algebras has been developed throughout many parts of the literature in various levels of generality. For example, the framework has been applied to: 
\begin{itemize}
\item homological mirror symmetry by M.\ Kontsevich \cite{Kont} and K.\ Fukaya \cite{Fuk};  
\item symplectic geometry by K.\ Fukaya, Y-G.\ Oh, H.\ Ohta, and K.\ Ono \cite{FOOO}; 
\item higher analogues of the Riemann-Hilbert correspondence by J.\ Chuang, J.\ Holstein, and A.\ Lazarev \cite{CHL}; 
\item the theory of higher geometric stacks by K.\ Behrend and E.\ Getzler \cite{BG}; 
\item Galois deformations by C.\ Wang-Erickson \cite{WE};    
\item $E_1$-deformation problems within the context of non-symmetric operads, by N.\ de Kleijn and F.\ Wierstra \cite{DW}.
\end{itemize}
We are not claiming that the above list is exhaustive. The last two works in particular provide much of the motivation for the present paper.

Following the terminology of J.\ Lurie \cite{HA}, and   Behrend--Getzler \cite{BG}, we define the {\it nerve} of a complete $A_\infty$-algebra $(A,d,\bm)$ to be the simplicial set
$\sNb(A):=\MC_{\Ass}(A \tensor N^\ast(\Del^\bl))$. Here, $N^\ast(\Del^\bl)$ is the simplicial dg algebra of normalized cochains. Note that different authors use different names for $\sNb(A)$, e.g.\ de Kleijn and Wierstra \cite{DW} call $\sNb(A)$ the ``simplicial Maurer-Cartan set'' of $A$. 

Indeed, the nerve is the $A_\infty$-analogue of the simplical Maurer-Cartan set and upgrades to a functor
\[ 
\sNb(-) \maps \cat{\wh{A}_{\infty}Alg} \to \sSet
\] 
from the category of complete $A_\infty$-algebras and $\infty$-morphisms to simplicial sets. In \cite{DW}, de Kleijn and Wierstra prove that $\sNb(A)$ is not just a simplical set, but a {\it Kan simplicial set} or {\it $\infty$-groupoid}. This is what one might expect, since the simplicial Maurer-Cartan set of an $L_\infty$-algebra has the same property \cite[Prop.\ 4.7]{Getzler}. 
On the other hand,  if $A$ is a finite-dimensional dg associative algebra, then $\sNb(A)$ 
is a simplicial object in the category of schemes over $\FF$, but it is  
not necessarily an $\infty$-groupoid. In fact, Behrend and Getzler \cite{BG} prove that, in this case, $\sNb(A)$ is only a {\it quasi-category} or {\it $(\infty,1)$-category} internal to schemes. 
In contrast with the pronilpotent case, this means that a 1-simplex in $\sNb(A)$ need not be homotopy invertible\footnote{The term ``quasi-invertible'' is used in \cite{BG} instead of ``homotopy invertible''. In the present paper, we use the term ``quasi-invertible'' in the sense of noncommutative ring theory. See Def.\ \ref{def:qi}.}. Furthermore, \cite[Thm.\ 7.5]{BG} implies that the maximal $\infty$-groupoid contained in $\sNb(A)$ is an open sub-simplicial scheme, i.e.\  homotopy invertibility is an ``open condition''. 
The results of Behrend and Getzler, when combined with those of de Kleijn and Wierstra, 
suggest an analogy to a more familiar geometric scenario: If $M$ is a monoid object in, say, finite type smooth schemes over $\FF$, then the inclusion $M^\times \to M$ of the group of invertible elements is an open immersion. Hence, if a closed point in $M$ is invertible, then its formal neighborhood is necessarily contained in the image of $M^\times \to M$.

\subsection*{Main results}   
In this paper, we further develop the Maurer-Cartan theory for complete $A_\infty$-algebras along the lines initiated by de Kleijn and Wierstra. In particular, our Theorems 1 and 3 below answer two questions posed by these authors in their work \cite{DW}. 
As a pilot application, in Sec.\ \ref{sec:def-app} we reinterpret  within our framework certain results of  C.\ Wang-Erickson \cite{WE} on the formal deformation theory of group representations. 
In future work, we plan to use the tools developed here to study noncommutative formality problems in positive characteristic. Such problems arise, for example, in Galois cohomology \cite{HW,Gal}, and in the study of configuration spaces \cite{Salvatore}.

We summarize our main results here in terms of $A_\infty$-algebras, even though 
we work with  ``shifted'' $A_\infty$-algebras throughout the rest of the paper. This makes certain calculations more straightforward; the differences are negligible, since we work with unbounded complexes. First, we develop some abstract homotopy theory in Sec.\ \ref{sec:cat-props} for the category $\cat{\wh{A}_{\infty}Alg}$, building on the second author's work \cite{CR}. (Some of the results in this section are the filtered versions of statements proved by K.\ Lefevre-Hasegawa in \cite{LH}.) In analogy with the complete $L_\infty$ case, we define in Def.\ \ref{def:cAsmap} an $\infty$-morphism 
$\ba{f} \maps (A,d,\bm) \to (A',d',\bm')$ between complete $A_\infty$-algebras  
to be a {\it weak equivalence} if the restriction $f_1 \vert_{\cF_nA}$ of $f_1$ to each piece of the filtration $(\cF_{n}A,d \vert_{\cF_nA})$ is a quasi-isomorphism of cochain complexes. 
This homotopical machinery allows us to give a concise proof of an $A_\infty$ analogue of the Goldman-Millson Theorem:

\begin{theorem1}[Thm \ref{thm:GM}]\label{thm:1}
If $\ba{f} \maps (A,d,\bm) \to (A',d',\bm')$ is a weak equivalence between complete $A_\infty$-algebras over a field $\FF$, then $\ba{f}$ induces a homotopy equivalence $\sNb(\ba{f}) \maps \sNb(A) \weq \sNb(A')$.
\end{theorem1}
\noindent Indeed, our proof of the above theorem is completely different than the one given by
V.\ Dolgushev and the second author in \cite{DR} for the $L_\infty$ Goldman-Millson Theorem. 

In \cite{Berglund}, A.\ Berglund characterizes the homotopy groups of $\sMCb(L)$ for a complete $L_\infty$-algebra $L$, in terms of the cohomology groups $H^{ < 0}(L)$ and the pronilpotent Lie algebra $H^0(L)$.
We prove the analogous result for $A_\infty$-algebras in Sec.\ \ref{sec:hmtpy}:
\begin{theorem2}[Thm.\ \ref{thm:pin}; Thm.\ \ref{thm:pi1}]\label{thm:2}
Let $A$ be a complete $A_\infty$-algebra over a field $\FF$. Then there are natural isomorphisms of abelian groups $\pi_n\bigl( \sNb(A),0) \cong  H^{1-n}(A)$, for all $n \geq 2$, and a natural isomorphism of groups
\[
\pi_1\bigl(\sNb(A),0) \cong H^0(A)^{\qm}
\]
where $H^0(A)^{\qm}$ denotes the group of quasi-invertible elements (Def.\ \ref{def:qi}) of the non-unital pronilpotent algebra $H^0(A)$.
\end{theorem2}

For our last main result, we set $\FF=\kk$ to be a field of characteristic zero.
Recall the classical fact that the commutator bracket gives any associative $\kk$-algebra the structure of a Lie algebra. T.\ Lada and M.\ Markl  showed in \cite{LM} that this generalizes to an analogous relationship between $A_\infty$-algebras and $L_\infty$-algebras: To any $A_\infty$-algebra $A$ one can construct an $L_\infty$-algebra $\cL(A)$ by skew-symmetrizing the $A_\infty$-products.
The construction is functorial with respect to ``strict'' $\infty$-morphisms of $A_\infty$-algebras. 
It is then natural to expect that there is a relationship between the nerve of $A$ and the simplicial Maurer-Cartan set of its commutator $L_\infty$-algebra $\cL(A)$.

\begin{theorem3}[Thm.\ \ref{thm:N-vs-MC}] \label{thm:3}
Let $A$ be a complete $A_\infty$-algebra over a field $\kk$ of characteristic zero, and let $\cL(A)$ be its commutator $L_\infty$-algebra. Then there is a span of weak homotopy equivalences of simplicial sets
\[
\sNb(A) \xleftarrow{\sim} S \xto{\sim} \sMCLb\bigl( \cL(A) \bigr)
\]
that is natural with respect to strict $\infty$-morphisms between $A_\infty$-algebras.
\end{theorem3}

\subsection*{Acknowledgments}
We thank Carl Wang-Erickson for explaining his work \cite{WE} to us, and calling our attention to the work of M.\ Hopkins and K.\ Wickelgren \cite{HW}. We thank Jesse Wolfson for conversations
about the work of K. Behrend and E.\ Getzler \cite{BG}. 
Finally, we are grateful for the careful reading of this manuscript by the anonymous referee, and we thank them for their feedback. The second author acknowledges support by a grant from the Simons Foundation/SFARI (585631,CR).

\section{Preliminaries} \label{sec:notation}
Throughout $\FF$ denotes a field of arbitrary characteristic, and $\kk$ is a field 
of characteristic $0$. In Sec.\ \ref{sec:char0}, we will set $\FF=\kk$. 
All graded objects are $\Z$-graded and unbounded, unless stated otherwise. 
We use cohomological conventions for differential graded (dg) objects.
The categories of graded vector spaces and cochain complexes over $\FF$ are denoted by $\VectF$ and $\ChF$, respectively. 
For $V$ a graded vector space, $sV:=V[-1]$ denotes the suspension of $V$, defined by $(sV)^i := V^{i - 1}$.  Similarly, $s^{-1}V:=V[1]$ is the desuspension of $V$ defined by $(s^{-1}V)^i := V^{i + 1}$.

\subsection{Notation for simplicial objects} \label{sec:simp-note}
We denote by $\Ord$ the category of ordinals $[0], [1], [2],$ etc., and for each $i=0,\ldots,n$, we let $d^i \maps [n-1] \to [n]$, and $s^{i} \maps [n+1] \to [n]$ denote the coface and codegeneracy maps in $\Ord$. Given a category $\catC$, the notation $\sC$ is reserved for the category of simplicial objects in $\catC$, i.e.\ the category of contravariant functors from $\Ord$ to $\catC$. 
We denote by $\Del^n:=\hom_{\Ord}(-,[n]) \in \sSet$ the standard simplicial $n$-simplex. The $i$th face, for example, of a $k$-simplex $\si \maps \Del^k \to \Del^n$ in $\Del^n$ is the $k-1$ 
simplex $d_i \si:= \si \cc d^i \maps \Del^{k-1} \to \Del^n$. The remaining face and degeneracy maps for $\Del^n$ are defined similarly.

The collection $\Del^\bl = \{\Del^n\}_{n \geq 0}$ forms a cosimplicial object in $\sSet$, and we write its coface and codegeneracy maps using the same notation as in the category $\Ord$, i.e.,   $d^i \maps \Del^{\bl-1} \to \Del^{\bl}$ and $s^j \maps \Del^{\bl +1} \to \Del^{\bl}$, respectively.
So, for example, if $\ta \maps \Del^{k} \to \Del^n$ is a $k$-simplex of $\Del^n$, then we denote its embedding into the $i$th face of $\Del^{n+1}$ as the map $d^i \ta:=  d^i \cc \ta \maps \Del^k \to \Del^{n+1}$.

\subsection{Normalized cochains} \label{sec:N}
Recall that the normalized chain functor $N_\ast \maps \sSet \to \ChainF$ sends a simplicial set $X$ to the normalized chain complex \cite{EZ} of the simplicial vector space $\FF X$ generated by $X$. The diagonal $X \to X \times X$  and constant map $X \to \ast$ give $\FF X$ the structure of a simplicial counital coassociative coalgebra. Together with the Alexander-Whitney map, this gives the chain complex $N_\ast(X)$ the structure of a counital dg coassociative coalgebra. The \df{normalized cochain complex} $N^\ast(X):= \Hom_\FF(N_\ast(X), \FF) \in \cChain(\FF)$ of a simplicial set $X$ is the degree-wise linear dual of $N_\ast(X)$. Dualizing the coalgebra structure on $N_\ast(X)$
provides $N^\ast(X)$ with the structure of a finite-dimensional unital dg associative algebra over $\FF$. We recall explicit formulas for this structure below in Sec.\ \ref{sec:basis}.

Moreover, the collection $N^\ast(\Del^\bl) := \{N^\ast(\Del^n)\}_{n \geq 0}$ forms a simplicial unital dg algebra. The simplicial structure is induced from the cosimplicial structure on $\Del^\bl$ described in Sec.\ \ref{sec:simp-note}. Throughout the paper, we let  $d_i \maps N^\ast(\Del^{\bl +1}) \to N^\ast(\Del^{\bl})$ and $s_j \maps N^{\ast}(\Del^{\bl}) \to N^{\ast}(\Del^{\bl +1})$ denote the induced face and degeneracy algebra homomorphisms, respectively.

\subsubsection{A basis for $N^\ast(\Del^\bl)$} \label{sec:basis}
We recall from \cite[\S 5]{DW} a convenient basis for  normalized cochains which we will use throughout the paper. The set of non-degenerate $k$-simplicies in $\Del^n$, which we denote by $\cB_{n,k}:=\{(0 \leq i_0 < i_1 <  \cdots < i_k \leq n) \in \Del^n\}$, forms a basis for the vector space $N_k(\Del^n)$. Hence $N^k(\Delta^n)$ is the span of the dual basis
\begin{equation} \label{eq:dual}
N^k(\Del^n) = \spann_\FF \bigl \{ \vphi_{\si} \st   \si \in \cB_{n,k}\}.
\end{equation}
In particular, we denote by
\[
\vphi_{[n]} \maps N_n(\Del^n) \to \FF
\]
the unique dual basis element of top degree $n$ which sends the non-degenerate $n$-simplex in $\Del^n$ to $1 \in \FF$. We also reserve the symbols
\[
\vph_{0}, \vph_{1}, \ldots, \vph_{n} \in N^0(\Del^n) 
\]
to denote the basis elements in degree zero, where $\vph_i$ takes the value of $1$ on the $i$th vertex of $\Del^n$. 

\subsubsection{Formulas for simplicial unital dga structure on $N^\ast(\Del^\bl)$} \label{sec:norm-form}
It will be useful to have explicit formulas for the unit, differential, cup product, and face maps of $N^\ast(\Del^\bl)$ in terms of the above basis. First, the unit is the morphism of simplicial algebras
\[
\FF \xto{\unit_\bl} N^\ast(\Del^\bl), \quad 1_\FF \mapsto \unit_n:= \sum_{i=0}^n \vph_i.
\]    
Here $\FF$ denotes, by abuse of notation, the constant simplicial dg algebra which, in each dimension, equals the ground field $\FF$ concentrated in degree 0. 

\begin{remark}\label{rmk:poincare}
The simplicial dga $N^\ast(\Del^\bl)$ satisfies the \textit{Poincar\'{e} Lemma}. That is, for each $n \geq 0$, the unit map $\unit_n$  induces a quasi-isomorphism of dg associative algebras.
\end{remark}

Next, the differential $\del \maps  N^k(\Delta^n) \to N^{k+1}(\Delta^n)$ can be written as
\begin{equation}\label{eq:diff}
\del (\vphi_\si) = \sum_{\ta \in \cB(n,k+1)} \sum_{i=0}^{k+1}  (-1)^{k+1+i} c_{\ta,i} \,  \vphi_\ta
\end{equation}
where $c_{\ta,i} = 1$ if $\si= d_i \ta$, and is zero otherwise. Also, given $\phi \in N^k(\Delta^n)$ and $\psi \in N^\el(\Delta^n)$, 
the multiplication or ``cup product'' $\ph \cupp \psi \in N^{k+\el}(\Del^n)$ is the linear map whose values on elements of $\cB_{n, k+ \el}$ are:
\begin{equation}\label{eq:cup}
\bigl( \phi \cupp \psi \bigr)(i_0 < i_1<  \cdots < i_{k+\el-1} < i_{k+\el} ) = \phi(i_0 < \cdots < i_k) \cdot \psi(i_k < \cdots i_{k+\el}).
\end{equation}
Let $j=0,\ldots,n$, and let $\vph_\si \in N^k(\Del^n)$ be a basis element. The $j$th-face $d_j \vph_\si \in N^{k}(\Del^{n-1})$  of $\vph_\si$ is:
\begin{equation} \label{eq:face}
d_{j}\vphi_\si = \sum_{\substack{ \ta \in \cB_{n-1,k} \\ d^{j} \ta = \si}} \vph_\ta.
\end{equation}
Obviously, the sum on the right hand side of Eq.\ \ref{eq:face} consists of at most one term, since the coface map $d^j$ is injective.

Finally, we note here for future reference that for all $n \geq 0$, the above formulas imply that the top degree cochain $\vph_{[n]} \in N^{n}(\Del^n)$ satisfies the equalities
\begin{equation} \label{eq:top}
\del \vph_{[n]}  = 0 = \vph_{[n]} \cupp \vph_{[n]}. \qquad d_j \vph_{[n]} =0 \quad \forall j=0,\ldots,n.
\end{equation}

\subsection{Filtered objects} \label{sec:filt}
Our main references for filtered dg objects and their completions are \cite[\S4.1; \S4.2]{CR}, 
\cite[\S7.3]{Fresse:BookI}, and \cite[Ch. 2]{DSV:2018}. Denote by $\fVectF$ the category of filtered vector spaces, i.e.\ vector spaces $V \in \VectF$ equipped with a descending filtration $$V=\cF_1V \supseteq \cF_2V \supseteq \cdots$$ 
beginning in filtration \und{degree $1$}. Morphisms in $\fVectF$ are linear maps which are filtration preserving. The category of complete filtered vector spaces is the full subcategory $\cVectF \sse \fVectF$ whose objects satisfy $V= \plim \qf{V}{n}$. Analogously, $\cChF$ denotes the category of $\Z$-graded complete filtered cochain complexes whose morphisms are filtration preserving chain maps.

\subsubsection{Morphisms} \label{sec:fmorph}
We recall a characterization of split epimorphisms in $\fVect$, that we will use throughout
Sec.\ \ref{sec:cat-props}.
\begin{lemma}[Lem.\ 4.3\cite{CR}]\label{lem:fepi}
If $f \maps V \to V'$ is a morphism in $\cVect$ such that the restriction
$\cF_nf \maps \cF_nV \to \cF_nV'$ is surjective  for all $n \geq 1$, then there exists a morphism $\si \maps V' \to V$ in $\cVect$ such that $f \si = \id_{V'}$.
\end{lemma}

\subsubsection{Symmetric monoidal structures} \label{sec:smc} 
Recall that the usual tensor product of graded vector spaces lifts to a symmetric monoidal 
structure on $\fVectF$ \cite[Lem.\ 2.5]{DSV:2018}. Indeed, if $V,W \in \fVectF$, then the canonical filtration \cite[\S7.3.9]{Fresse:BookI} on $V \tensor W$ is
\begin{equation} \label{eq:filt-tensor}
\cF_n(V \tensor W) = \sum_{\el =1}^n \cF_{\el}V \tensor \cF_{n-\el+1}W \qquad \forall n \geq 1.
\end{equation}
As a special case, if $W \in \VectF$ is a graded vector space equipped with the discrete filtration, then the filtration above on $V \tensor W$ reduces to $\cF_n(V \tensor W) = \cF_nV \tensor W.$
Hence, the ``discrete functor'' $\Disc \maps \VectF \to \fVectF$ is strongly symmetric monoidal \cite[Prop.\ 2.8]{DSV:2018}.

If $ V, W \in \cVect \sse \fVect$ are complete, then the completed tensor product of $V, W \in \cVect$ is
\begin{equation} \label{eq:comp-tensor}
V \ctensor W := \plim \bigl ( V \tensor W/ \cF_n(V \tensor W) \bigr).
\end{equation}
Again, as a special case, if $W \in \Vect$ is given the discrete filtration, then there is a canonical identification
\begin{equation} \label{eq:ctensor-disc}
V \ctensor W = \plim (V/\cF_nV) \tensor W. 
\end{equation}
The complete tensor product gives $\cVectF$ the structure of a symmetric monoidal category \cite[\S7.3.12]{Fresse:BookI}, and the completion functor $(-)\wh{~~} \maps \fVect \to \cVect$ 
is strongly symmetric monoidal. Indeed, given filtered vector spaces $V, W \in \fVect$,  there is a natural isomorphism \cite[Prop.\ 2.21]{DSV:2018} \cite[Prop.\ 7.3.11]{Fresse:BookI} in $\cVectF$:
\begin{equation} \label{eq:ctensor-iso}
(V \tensor W) \wh{~~} \xto{\cong} \wh{V} \ctensor \wh{W}.
\end{equation}   
Furthermore, the composition of the completion with the discrete functor $\Disc(-)$ 
\begin{equation} \label{eq:comp-disc}
\wh{\Disc} \maps \Vect \to \cVect
\end{equation}
is also a strongly symmetric monoidal functor. \cite[Cor.\ 2.22]{DSV:2018}.

Clearly all of the above results remain valid if we replace the symmetric monoidal categories $\VectF$ and $\cVectF$ with $\ChF$ and $\cChF$, respectively.

\subsection{Homotopy theory in $\cCh(\FF)$} \label{sec:cCh}
We recall facts concerning homotopy theory in the category $\cCh(\FF)$. We follow \cite[\S4.3]{CR} and the references given in op.\ cit. A morphism $f \maps (V,d) \to (V',d')$ is a \df{weak equivalence}/\df{fibration} if the restriction of $f$ to the subcomplexes
\[
\cF_n(f):= f \vert_{\cF_nV} \maps (\cF_nV,d) \to (\cF_nV',d') \qquad \forall n \geq 1
\] 
is a quasi-isomorphism/epimorphism of cochain complexes. We say $f$ is an \df{acyclic fibration} if $f$ is both a weak equivalence and a fibration. Note that weak equivalences in $\cChF$ satisfy the \df{2 of 3 property}: If $f \maps (V,d) \to (V',d')$ and 
$g \maps (V',d') \to (V'',d'')$ are composable morphisms in $\cChF$, and any two of the morphisms
\[
f, \quad g, \quad g\cc f
\]
are weak equivalences, then so is the third.

The proof of the following proposition is straightforward, but the result is crucial. It shows that every acyclic fibration in $\cChF$ admits a right inverse in $\cChF$. This right inverse is necessarily a weak equivalence by the 2 of 3 property. 
\begin{proposition}[Prop.\ 4.13, \cite{CR}] \label{prop:acyc}
Let $f \maps (V,d) \afib (V',d')$ be an acyclic fibration in $\cChF$. Then
\begin{enumerate}
\item There exists a morphism $\tau \maps (V',d') \to (V,d)$ in $\cChF$  such that $f\tau = \id_{V'}$.   
\item There exists a degree $-1$ linear map $h \maps V \to \ker f$ such that $h(\cF_nV) \sse \ker f \cap \cF_nV$ for all $n \geq 1$, and such that $\id - \tau f = dh + hd$.
\end{enumerate} 
\end{proposition}
 
Another needed result is that the completed tensor product in $\cChF$, as defined above, is compatible with quasi-isomorphisms between ``discrete'' cochain complexes.
\begin{lemma}[Lem.\ 4.12 \cite{CR}] \label{lem:ctensor}
Let $(V,d_V) \in \cChF$ be a complete cochain complex. If $f \maps (B,\del) \weq (B',\del')$ is a quasi-isomorphism in $\ChF$, then 
$\id_V \ctensor f \maps V \ctensor B \to V \ctensor B'$ is a weak equivalence in $\cChF$.   
\end{lemma}

\subsubsection{Models for intervals} \label{sec:int-model}
Recall from \cite[\S4.3.2]{CR} that 
a unital dg associative $\FF$-algebra $\J:=(\J,\del, \mult,\unit) \in \dgaF$ is a \df{model for the interval} if there exist unital dg algebra morphisms $\ev_0,\ev_1 \maps \J \to \FF$ such that the composition $\FF \xto{\unit} \J \xto{\ro} \FF \times \FF$ is the diagonal, where
$\ro:=(\ev_0,\ev_1)$. Furthermore, we require  $\rho$ to be an epimorphism, and both $\ev_0$ and $\ev_1$ to be  quasi-isomorphisms of cochain complexes.

Note that the axioms above imply that the unit $\unit$ is also a  quasi-isomorphism, and that 
the ``evaluation maps'' $\ev_0$, $\ev_1$ are epimorphisms. 
\begin{example}\label{ex:interval}
The main example for us is
$\J=N^\ast(\Del^1)$, the dg algebra \eqref{sec:N} of normalized cochain  on the 1-simplex.
Here  $\ev_0, \ev_1 \maps N^\ast(\Del^1) \to \FF$ are defined on the basis \eqref{sec:basis} to be 
\begin{equation} \label{eq:eval}
\ev_0(\vphi_0):= \ev_1(\vphi_1):=1_\FF, \quad \ev_0(\vphi_1):= \ev_1(\vphi_0):=0.
\end{equation} 
Another important example, for the case when $\FF=\kk$ is a field of characteristic zero,
is the commutative graded algebra  $\J:=\Om^\ast(\Del^1)$ of polynomial de Rham forms  on the 1-simplex.
This algebra will appear later in Section \ref{sec:char0}.
\end{example}

Models for the interval provide the category $\cChF$ with a way to factor morphisms. 
In the following proposition, we may choose $\J$ to be either of the two examples, depending on the characteristic of $\FF$. 
\begin{proposition}[Prop.\ 4.11 \cite{CR}] \label{prop:Ch-pobj}
Let $(V,d) \in \cChF$, and  let $\J:=(\J,\del, \mult,\unit)$ be a model for the interval.
Then the composition
\[
V \xto{\id \tensor \unit} V \ctensor \J \xto{\id_V \ctensor \ro} V \times V
\] 
is a factorization in $\cChF$ of the diagonal map $V \to V \times V$. Furthermore, $\id_V \ctensor \unit$ is a weak equivalence, and $\id_V \ctensor \ro$ is a fibration.
\end{proposition}

\section{Complete filtered $\sainf$-algebras} \label{sec:comp-Ass}

\subsection{$\sainf$-algebras and $\infty$-morphisms} \label{sec:ainf}
Our notation and conventions in this section follow, for the most part,  \S2 and  \S3.1 of \cite{DW}. In particular, we refer the reader there for background on dg coassociative coalgebras. 
Note that, in contrast with \cite{DW}, we do not consider ``curved'' structures in this paper. (See Rmk.\ \ref{rmk:curved} for further discussion of this point.) Recall that the cofree conilpotent coassociative coalgebra generated by $V \in \VectF$ is the graded vector space $\Tc(V):= \T(V)=\bigoplus^{\infty}_{k = 1} V^{\tensor k}$ equipped with the comultiplication corresponding to ``deconcatenation of tensor words''.  An \df{$\sainf$-algebra} or {``shifted $\ainf$-algebra''} over $\FF$ is a triple  $\AdQ{}$ consisting of  a cochain complex $(A, d) \in \ChF$ equipped with a degree $1$ codifferential $Q$ on $\Tc(A)$ such that 
$Q\vert_A =d$ and $Q^2 =0$. An \df{$\infty$-morphism} $\Ph \maps \AdQ{} \to \AdQ{\prime}$ 
between $\sainf$-algebras is a dg coalgebra morphism $\Ph \maps (\Tc(A),Q) \to  (\Tc(A'),Q')$. We denote by $\As$ the category of $\sainf$-algebras and $\infty$-morphisms.

Since $\Tc(V)$ is ``cofree'' for any graded vector space $V$, every degree 1 coderivation $Q \maps \Tc(V) \to \Tc(V)$ is uniquely determined by the projection $Q^1:= \mathrm{pr}_{V} \circ Q \maps \T(V) \to V$. The linear map $Q^1$ is, in turn, uniquely determined by its degree 1 ``structure maps''
\begin{equation} \label{eq:Qstruct}
\sQ_m:= \sQ \vert_{V^{\tensor m}} \maps V^{\tensor m} \to V, \quad \forall m \geq 1.
\end{equation}  
Note that if $\AdQ{} \in \As$ is a $\sainf$-structure on the complex $(A,d)$, then $\sQ_1$ is necessarily the differential $d$ on $A$. Furthermore, the equality $Q^2 =0$ is equivalent to the  equalities $\sum_{k=1}^{n} Q^1_k Q^k_n =0$ for all $n\geq 1$, where $Q^k_n:= \pr_{A^{\tensor k}} \cc Q \vert_{A^{\tensor n}}$.

Similarly, every graded coalgebra morphism $\Ph \maps \Tc(V) \to \Tc(V')$ is uniquely determined by the degree 0 linear map $\Ph^1:= \pr_{V'} \cc \Ph \maps \T(V) \to V'$, which is itself determined by the structure maps:
\begin{equation} \label{eq:Fstruct}
\Ph^1_m:= \Ph^1 \vert_{V^{\tensor m}} \maps V^{\tensor m} \to V', \quad \forall m \geq 1.
\end{equation}
In particular, if $\Ph \maps \AdQ{} \to \AdQ{\prime}$ is an $\infty$-morphism, then the compatibility with the codifferentials $ \Ph Q=Q' \Ph$ is equivalent to the  equalities $\sum_{k=1}^{n} \Ph^1_k Q^k_n =
\sum_{k=1}^{n} Q^{'1}_k \Ph^k_n$ for all $n\geq 1$, where $\Ph^k_n:= \pr_{A^{\prime \tensor k}} \cc \Ph \vert_{A^{\tensor n}}$. For $n=1$, this implies that the linear map $\sPh$ satisfies $d' \cc \sPh = \sPh \cc d$, i.e.\ it is a morphism between the cochain complexes $(A,d)$ and $(A',d')$. The \df{tangent functor} 
is the functor 
\begin{equation} \label{eq:tan}
\tan \maps \As \to \ChF
\end{equation}
which assigns an $\sainf$-algebra $\AdQ{}$ to its tangent cochain complex $(A,d)$, and an $\infty$-morphism $\Phi\maps \AdQ{} \to \AdQ{\prime}$
to its linear term $\sPh \maps (A,d) \to (A',d')$. 

\begin{example} \label{ex:dga}
Let $(A,d,\mu) \in \dgaF$ be a (non-unital) dg associative algebra. Then $(\bs^{-1} A, \ti{d}, Q) \in \As$ where: $\bs^{-1} A=A[1]$ is the desuspension of $A$; $\ti{d}:=\bs^{-1} \cc d \cc \bs$; and $Q$ is the unique degree 1 coderivation such that $Q^1_1 = \ti{d}$, $Q^1_2 = \bs^{-1} \cc \mu \cc (\bs \tensor \bs)$, and $Q^1_{k \geq 3} =0$. Moreover, a morphism $\phi$ of dg algebras defines a unique $\infty$- morphism $\Ph$ between the corresponding $\sainf$-algebras with $\sPh= \bs^{-1} \cc \ph \cc \bs$, and $\Ph^{1}_{k\geq 2}=0$. 
\end{example}

Finally, recall that a morphism $\Phi \maps \AdQ{} \to \AdQ{\prime}$ in $\As$ is \df{strict} if $\Ph^1_k=0$ for all $ k\geq 2$, and we say $\Ph$ is an \df{$\infty$-quasi-isomorphism} if the tangent map $\tan(\Ph)$ is a quasi-isomorphism of cochain complexes.

\subsection{The category of complete filtered $\ainf[1]$-algebras} \label{sec:comp-ainf}

Following, for example, \cite[\S3.1]{DW}, a \df{complete  $\sainf$-algebra} is a complete cochain complex $(A,d) \in \cChF$ equipped with an $\sainf$-structure  
$\AdQ{}$ which is compatible with the filtration, i.e.\ for all $n \geq 1$
\[
\sQ_n(\cF_{i_1}A \tensor \cF_{i_2}A \tensor \cdots \tensor \cF_{i_n}A) \sse \cF_{i_1 + i_2 + \cdots + i_n}A.
\]
A  \df{(filtered) $\infty$-morphism} $\Ph \maps \AdQ{} \to \AdQ{\prime}$ between  $\csainf$-algebras is an
$\infty$-morphism such that for all $n \geq 1$
\[
\Ph^1_n(\cF_{i_1}A \tensor \cF_{i_2}A \tensor \cdots \tensor \cF_{i_n}A) \sse  \cF_{i_1 + i_2 + \cdots + i_n}A'.
\]

Denote by $\cAs$, the category of   $\csainf$ algebras. The tangent functor \eqref{eq:tan} extends to the complete filtered setting. We also denote this functor as $\ctan \maps \cAs \to \cCh$. 
\begin{definition} \label{def:cAsmap}
A morphism $\Ph \maps \AdQ{} \to \AdQ{\prime}$ in $\cAs$ is a \df{weak equivalence/}\\
\df{fibration/acyclic fibration} if $\ctan(\Ph)$ is a weak equivalence/fibration/acyclic fibration in $\cCh$.
\end{definition}

\subsection{Tensoring   $\csainf$-algebras with dg algebras} \label{sec:Ainf-tensor}
Continuing the discussion on tensor products from \S\ref{sec:smc}, if $\AdQ{}$  is an   $\csainf$-algebra, and $B=(B,\del,\cdotp) \in \dga$ is a dg associative algebra, then the tensor product $(A \ctensor B, d_{B}, Q_{B})$ is a   $\csainf$-algebra
with $d_B := d \tensor \id_B + \id_A \tensor \del$, and
\begin{equation} \label{eq:Ainf-tensor1}
(Q_{B})^1_{k}(x_1 \tensor b_1, x_2 \tensor b_2, \cdots, x_k \tensor b_k):=
(-1)^{\varepsilon}\sQ_k(x_1,\ldots,x_k) \tensor b_1 \cdotp b_2
\cdotp \cdots \cdotp b_k
\end{equation}
for $k \geq 2$, where $(-1)^\varepsilon$ is the usual Koszul sign. 
Note that this construction defines a functor 
\begin{equation} \label{eq:Ainf-tensor2}
A \ctensor - \maps \dga \to \cAs.
\end{equation}
Indeed, if $\phi \maps B \to B'$ is a morphism of dg algebras, then  
$\id_A \ctensor \phi$ is a strict $\infty$-morphism in $\cAs$. 

Similarly, given $B \in \dga$ and a morphism $\Ph \maps \AdQ{} \to \AdQ{\prime}$ in $\cAs$, we obtain an $\infty$-morphism between the tensor products $\Ph_B \maps (A \ctensor B, d_{B}, Q_{B}) \to (A' \ctensor B, d'_{B}, Q'_{B})$ where for each $k \geq 1$
\[
(\Ph_{B})^1_{k}(x_1 \tensor b_1, x_2 \tensor b_2, \cdots, x_k \tensor b_k):=
(-1)^{\varepsilon}\Ph^1_k(x_1,\ldots,x_k) \tensor b_1 \cdotp b_2
\cdotp \cdots \cdotp b_k.
\]
This defines a functor
\begin{equation} \label{eq:Ainf-tensor3}
-\ctensor B \maps \cAs \to \cAs.
\end{equation}

The next lemma will play an important role in \S\ref{sec:char0}. 
\begin{lemma}\label{lem:ainf-tensor}
Let $B,C \in \dga$ be dg associative algebras, and denote by $B \tensor C \in \dga$ their usual tensor product as dg algebras. Let $\AdQ{} \in \cAs$. Then there are strict isomorphisms in $\cAs$ 
\begin{equation} \label{eq:ainf-tensor2}
\begin{split}
\bigl( (A \ctensor B) \ctensor C, (d_{B})_{C} , (Q_{B})_{C} \bigr) &\cong
\bigl( A \ctensor (B \tensor C), d_{B \tensor C}, Q_{B \tensor C} \bigr)\\ 
&\cong \bigl( A \ctensor (C \tensor B), d_{C \tensor B}, Q_{C \tensor B} \bigr) \cong
\bigl( (A \ctensor C) \ctensor B, (d_{C})_{B} , (Q_{C})_{B} \bigr) 
\end{split}
\end{equation} 
which are natural in $A$, $B$, and $C$.
\end{lemma}
\begin{proof}
The morphisms \eqref{eq:ainf-tensor2} originate from the natural isomorphisms in $\Vect$ and $\cVect$ which witness the associativity and braiding of the respective symmetric monoidal structures.
Hence, as morphisms between complete filtered vector spaces, they are indeed isomorphisms since the functor \eqref{eq:comp-disc} is strong symmetric monoidal. All that remains is to verify that  
the functors 
\[
\dga \xto{(A \ctensor B) \ctensor -} \cAs, \qquad
\dga \xto{B \tensor -} \dga \xto{A \ctensor - } \cAs
\]
are naturally isomorphic as functors into   $\csainf$-algebras. This follows by direct inspection
of the formulas \eqref{eq:Ainf-tensor1} for the codifferentials $(Q_{B})_{C}$    
and $Q_{B \tensor C}$.
\end{proof}

\subsection{Maurer-Cartan elements} \label{sec:MC}
The standard facts concerning the Maurer-Cartan theory of $L_\infty$-algebras in characteristic zero, e.g., \cite[\S 5.4]{CR}, readily extend to  $A_\infty$-algebras over $\FF$. Our reference for this section is \cite[\S 4]{DW}. Given $\AdQ{} \in \cAs$, the \df{curvature}  $\curv_\Ass \maps A^0 \to A^{1}$ is the function which assigns to a degree zero element $a \in A^0$ the infinite series $\curv_\Ass(a):=\sum_{n=1}^{\infty} \sQ_n(a^{\tensor n})$. This is well defined since $A$ is complete and $A=\cF_1A$. 

The \df{Maurer-Cartan set} of $\AdQ{}$ is the set  $\MCAs(A):=\{a \in A^0 \st \curv_\Ass(a)=0\}$. 
An $\infty$-morphism $\Ph \maps \AdQ{} \to \AdQ{\prime}$ in $\cAs$ yields a set-theoretic function
\begin{equation} \label{eq:phistar}
\Ph_\ast \maps A^0 \to A^{\prime  0}, \quad \Ph_\ast(a):= \sum_{n=1}^{\infty} \Ph^1_n(a^{\tensor n})
\end{equation} 
which is well-defined due to the  compatibility of $\Phi$ with the filtrations. Note that if $\Ph=\sPh$ is strict, then $\Ph_\ast(a)=\sPh(a)$. From the equality
\[
\curv_{\Ass}' \bigl(\Ph_\ast(a) \bigr) = \sum_{n \geq 0}^{\infty} \Ph^1_{n+1} ( \shpr{a^{\tensor n}}{\curv_{\Ass}(a)} ),
\]
where 
\begin{equation} \label{eq:shpr}
\star_{\ssh} \maps \Tc(A) \tensor \Tc(A) \to \Tc(A)
\end{equation}
is the usual shuffle product \cite[\S2]{DW}, it follows that $\Ph_\ast$ descends to a function between Maurer-Cartan sets. The \df{Maurer-Cartan functor} $\MCAs(-) \maps \cAs \to \Set$ sends
 a morphism $\Ph \maps \AdQ{} \to \AdQ{\prime}$ in $\cAs$ to the function $\MCAs(\Ph):=\Ph_\ast \maps \MCAs(A) \to \MCAs(A')$. 
\subsubsection{Twisting} \label{sec:twist}
Given a degree 0 element $a \in A^0$ of a   $\csainf$-algebra $\AdQ{}$, we define for each $n \geq 1$ a ``twisted'' $n$-ary operation $(Q^a)^1_n \maps A^{\tensor n} \to A$ via the formula
\[
(Q^a)^1_n(x_1,\ldots,x_n):= \sum_{k\geq 0}^{\infty} \sQ_k\bigl(\shpr{a^{\tensor k}}{(x_1 \tensor x_2 \tensor \cdots \tensor x_n)} \bigr). 
\]
As explained in \cite[Lem.\ 4.3, Def.\ 4.4]{DW}, if $\al \in \MCAs(A)$ is a Maurer-Cartan element, then the operations $(Q^\al)^1_n$ induce a new   $\csainf$-algebra structure $A^\al:=(A,d^\al,Q^\al)$ on the original underlying graded vector space $A$. The twisted differential on the underlying tangent cochain complex $\ctan(A^\al)=(A,d^\al) \in \cChF$ is necessarily defined to be $d^a:=(Q^a)^1_1$. Moreover, $\infty$-morphisms can be twisted as well \cite[Def.\ 4.9]{DW}.
Given  $\Ph \maps \AdQ{} \to \AdQ{\prime}$ in $\cAs$, and $\al \in \MCAs(A)$ the linear maps
\[
(\Ph^\al)^1_n(a_1, \ldots, a_n):= \sum^\infty_{k=0} \Ph^1_{n+k} \bigl( \shpr{\al^{\tensor k}}
{(a_1 \tensor \cdots a_n)} \bigr)
\] 
define an $\infty$-morphism $\Ph^\al \maps (A^\al,d^\al,Q^{\al}) \to (A^{\prime \be},d^{\prime \be},Q^{\prime \be})$, where $\be:=\Ph_{\ast}(\al)\in \MCAs(A')$.

\section{The nerve of a complete $\sainf$-algebra} \label{sec:nerve}
The  nerve $\sNb(A)$ of a   $\csainf$-algebra $\AdQ{} \in \cAs$ is a simplicial set that generalizes the classical Dold-Kan correspondence for non-positively graded cochain complexes. Here we recall the construction of $\sNb(A)$ and its basic properties. Note that in  \cite{DW}, de Kleijn and Wierstra call $\sNb(A)$ the ``Maurer-Cartan simplicial set of $A$''. We reserve that terminology for a different simplicial set which appears later in Sec.\ \ref{sec:char0}.

\subsection{The functor $\sNb(-)$} \label{sec:nerve-kan}
Given $\AdQ{} \in \cAs$, for all $n\geq 0$, the complete tensor product
$A \ctensor N^\ast(\Del^n)$ is a $\csainf$-algebra, as described in Sec.\ \ref{sec:Ainf-tensor}. Furthermore, since $N^\ast(\Del^n)$ is a finite-dimensional $\FF$-vector space, the completed tensor product is naturally isomorphic to the usual tensor product over $\FF$:
\[
A \ctensor N^\ast(\Del^n) \cong  A \tensor N^\ast(\Del^n).
\]
\begin{definition} \label{def:nerve}
The \df{nerve} of a   $\csainf$-algebra $\AdQ{} \in \cAs$ is the simplicial set $\sNb(A)$, where
\[
\sN_n(A):= \MCAs \bigl(A \tensor N^\ast(\Del^n)\bigr) \qquad \forall n\geq 0.
\]
\end{definition}
Composing the functor \eqref{eq:Ainf-tensor3} $- \tensor N^\ast(\Del^n) \maps \cAs \to \cAs$  with
the Maurer-Cartan functor $\MCAs \maps \cAs \to \Set$ shows that the nerve is a functorial construction
\[
\sNb \maps \cAs \to \sSet.
\]  
An initial study of the homotopical properties of the functor $\sNb(-)$ was established by de Kleijn and Wierstra \cite{DW}. In the notation of the present paper, they proved the following:  

\begin{theorem}[Prop.\ 6.6, Cor.\ 6.7; \cite{DW}]\label{thm:Kan}
Let $\Ph \maps \AdQ{} \to \AdQ{\prime}$ be a strict fibration in $\cAs$. Then
$\sNb(\Ph) \maps  \AdQ{} \to \AdQ{\prime}$ is a Kan fibration of simplicial sets. In particular, $\sNb(A)$ is a Kan simplicial set for every $\AdQ{} \in \cAs$, and so the nerve functor equals its corestriction
\[
\sNb \maps \cAs \to \Kan.
\]
to the full subcategory $\Kan \sse \sSet$.
\end{theorem}

\subsection{Basepoints of $\sNb(A)$ and twisting} \label{sec:basepoints}
The twisting procedure recalled in Sec.\ \ref{sec:twist} provides a convenient way to switch between basepoints in the nerve. The next lemma is the $A_\infty$ version of \cite[Lem.\ 4.3]{DR-shift}.  
The proof is exactly the same.
\begin{lemma}\label{lem:shift}
Let $\AdQ{} \in \cAs$, $\al \in \MCAs(A)$, and $(A^\al, d^\al, Q^\al)$ be the corresponding $\csainf$-algebra
twisted by $\al$. Then the assignment $\beta \in \sN_n(A^{\al}) \mapsto \al + \beta \in \sN_n(A)$
induces an isomorphism of simplicial sets $\mathrm{Shift}_{\al}\maps \sNb(A^{\al}) \xto{\cong} \sNb(A)$.
Moreover, for each morphism $\Ph \maps \AdQ{} \to \AdQ{\prime}$ in $\cAs$, the following diagram commutes:

\[
\begin{tikzdiag}{2}{4}
{
\sNb(A^{\al})\& \sNb(A)  \\
\sNb(A^{\prime \, \Phi_{*}(\al)}) \&\sNb(A^{\prime}) \\
};

\path[->,font=\scriptsize]
(m-1-1) edge node[auto] {$\mathrm{Shift}_{\al}$} (m-1-2)
(m-2-1) edge node[auto] {$\mathrm{Shift}_{\Phi_{*}(\al)}$} (m-2-2)
(m-1-1) edge node[auto,swap] {$\sNb(\Phi^\al)$} (m-2-1)
(m-1-2) edge node[auto] {$\sNb(\Phi)$} (m-2-2)
;
\end{tikzdiag}
\]
\end{lemma}

\section{Homotopical properties} \label{sec:cat-props}
Here we further develop the abstract homotopy theory for the category $\cAs$  in order to give a concise proof of the $\csainf$ analog of the Goldman-Millson Theorem in Sec.\ \ref{sec:GM}. 

\subsection{Finite products} \label{sec:products}
We denote the product of $\AdQ{}$ and $\AdQ{\prime}$ in $\cAs$  by $(A\times A', d_\times,Q_\times)$, where $(A\times A',d_\times)$ is the product of complete cochain complexes, and $Q_\times$ is the degree one codifferential on $\T(A \times A')$ whose structure maps are
\begin{equation} \label{eq:product1}
(Q_\times)^1_{k} \bigl( (x_1, x'_1), (x_2,x'_2), \ldots, (x_k,x'_k) \bigr):= \Bigl(\sQ_{k}(x_1,x_2,\ldots,x_k) ,\spQ_{k}(x'_1,x'_2,\ldots,x'_k) \Bigr).
\end{equation}
The usual projections $\pr \maps A \times A' \to A$, $\pr' \maps A \times A' \to A'$ in $\cChF$ lift to strict fibrations $\ppr$ and $\ppr'$ in $\cAs$. Given morphisms $\Ph \maps \AdQ{\prime \prime} \to \AdQ{}$ and $\Ph' \maps \AdQ{\prime \prime} \to \AdQ{\prime}$ in $\cAs$, the unique morphism 
$\ti{\Ph} \maps \AdQ{\prime \prime} \to (A\times A', d_\times,Q_\times)$ into the product has the structure maps
\[
\ti{\Ph}^1_k(a_1, a_2, \ldots,a_k)= \bigl( \Ph^1_k(a_1,\ldots,a_k), \Ph^{\prime 1}_k(a_1,\ldots,a_k) \bigr). 
\] 
Given morphisms $\Ph_1 \maps (A_1,d^{A_1},Q^{A_1}) \to (A'_1,d^{A'_1},Q^{A'_1})$ and $\Ph_2 \maps (A_2,d^{A_2},Q^{A_2}) \to (A'_2,d^{A'_2},Q^{A'_2})$, we denote by 
\begin{equation} \label{eq:product2}
\Ph_1 \times \Ph_2 \maps (A_1 \times A_2, d_\times,  Q_{\times}) \to (A'_1 \times A'_2, d'_\times,Q'_{\times})
\end{equation}
the unique morphism in $\cAs$ between the corresponding products that satisfies the usual universal property.
Also, for any $\AdQ{}$ and $\AdQ{\prime}$ in $\cAs$, the inclusion map 
of complete vector spaces $A \emb A \times A'$ extends to a strict $\infty$-morphism in $\cAs$
\begin{equation} \label{eq:Ainf-incl}
\AdQ{} \to (A \times A', d_\times,Q_\times).
\end{equation}

Finally, given $\AdQ{}$ and $\AdQ{\prime}$ in $\cAs$,
it follows from Eq.\ \ref{eq:product1} and the definition \eqref{eq:Ainf-tensor2} of the
functor  $(A \times A') \ctensor - \maps \dga \to \cAs$ that there is a canonical isomorphism of simplicial $\csainf$-algebras
\[
(A \times A') \tensor N^\ast(\Del^\bl) \cong  A\tensor N^\ast(\Del^\bl) \,  \times \,
A\tensor N^\ast(\Del^\bl) \quad \text{in $s\cAs$}.
\]
Via this isomorphism, it is easy to see that the nerve functor $\sNb \maps \cAs \to \Kan$ preserves products:
\begin{proposition} \label{prop:Nprod}
The map of simplicial sets
\[
\bigl( \sNb({\ppr}), \sNb(\ppr') \bigr)  \maps \sNb(A \times A') \to \sNb(A) \times \sNb(A')
\]
is an isomorphism.
\end{proposition}

\subsection{Decomposition and right inverses of acyclic fibrations} \label{sec:acyc}
We show that the domain of any acyclic fibration in $\cAs$ is isomorphic to a direct product. As a result, every acyclic fibration in $\cAs$ admits a right inverse.

\begin{proposition}\label{prop:min-mod}
Let $\Ph \maps (A, d, Q) \afib (A',d',Q')$ be an acyclic fibration in $\cAs$.
Let $(\ker \sPh, d)$ denote the kernel of the chain map $\sPh \maps (A, d) \to (A',d')$ considered as an abelian $\sainf$-algebra. Then there exists an $\infty$-morphism in $\cAs$
\[
\Psi \maps (A, d, Q) \to (\ker \sPh, d)
 \]
such that the morphism induced via the universal property of the product:
\begin{equation} \label{eq:min_mod_iso}
\bigl(\Ph,\Psi \bigr) \maps (A,d,Q) \to (A' \tim \ker \sPh, d_\tim, Q_\tim)
\end{equation}
is an isomorphism in $\cAs$.
\end{proposition}

\begin{proof} 
We follow the proof of the analogous statement \cite[Lem.\ 6.8]{CR} for complete $\sinf$-algebras in characteristic zero.  Since $\Phi$ is an acyclic fibration, the chain map $\sPh \maps (A, d) \to (A',d')$
is an acyclic fibration in $\cChF$. Statement 2 of Prop.\ \ref{prop:acyc} implies that 
there exists a filtered chain map $\tau \maps (A',d') \to (A,d)$ in $\cChF$ and a filtered chain homotopy $h \maps A \to A[-1]$ such that
\begin{equation} \label{eq:min_mod1}
\sPh \cc \tau = \id_{A'}, \quad \sPsi_1 = dh + hd.
\end{equation}
where  $\sPsi_1 \maps (A,d) \to (\ker \sPh , d)$ is the chain map $\sPsi_1 := \id_A - \tau \cc \sPh$.
For each $n \geq 2$, let $\sPsi_n \maps \T^n(A) \to \ker \sPh$ be the linear map
\begin{equation} \label{eq:min_mod1.1.2}
\sPsi_{n}:= \sPsi_1 \circ h \circ \sQ_n.
\end{equation}
Since $\ta$, $h$ and $\sQ_k$ are filtration preserving, so is $\sPsi \maps \T(A) \to \ker \sPh$. A direct calculation using \eqref{eq:min_mod1} and the fact that $Q^2=0$, shows
that for each $n \geq 1$
\[
d \vert_{\ker} \circ \sPsi_n = \sPsi_1 Q^1_n + \sum_{k=2}^n \sPsi_k Q^k_n.
\]
Hence $\sPsi$ is an $\infty$-morphism in $\cAs$.
Finally, we note that the linear map 
\[
\tha \maps  A' \tim \ker \sPh \to A, \quad \tha(x',z):= \ta(x') + z
\]
is filtration preserving and inverse to $(\sPh,\sPsi_1)$ in $\cVect$. Hence, 
$(\Phi, \Psi)$ is an isomorphism in $\cAs$.
\end{proof}

The above proposition implies that every acyclic fibration can be ``strictified''.

\begin{corollary}\label{cor:min-mod}
Let $\Ph \maps (A, d, Q) \afib (A',d',Q')$ be an acyclic fibration in $\cAs$.
Then, in the notation of Prop. \ref{prop:acyc}, the following diagram of acyclic fibrations commutes 
\begin{equation} \label{diag:acyc-strict}
\begin{tikzdiag}{2}{4}
{
(A, d, Q) \& (A' \tim \ker \sPh, d_\tim, Q_\tim)\\
(A',d',Q') \& (A',d',Q')\\
};

\path[->,font=\scriptsize]
(m-1-1) edge node[auto] {$(\Phi,\Psi)$} node[below]{$\cong$} (m-1-2)

(m-2-1) edge node[auto] {$\id_{A'}$} (m-2-2)
;
\path[->>,font=\scriptsize]
(m-1-1) edge node[auto,swap] {$\Phi$} node[sloped,above]{$\sim$} (m-2-1)
(m-1-2) edge node[auto] {$\Pr_{A'}$} node[sloped,below]{$\sim$} (m-2-2)
;
\end{tikzdiag}
\end{equation}
\end{corollary}
\begin{proof}
The diagram commutes by construction. Applying the tangent functor $\ctan \maps \cAs \to \cChF$ gives a commutative diagram in which the top, bottom, and left hand side morphisms are all weak equivalences.
Hence, $\ctan(\Pr_A')$ is a weak equivalence in $\cChF$ as well, and so, by definition, $\Pr'_A$ is a strict acyclic fibration in $\cAs$.
\end{proof} 

Furthermore, Prop.\ \ref{prop:min-mod} implies that  every acyclic fibration in $\cAs$ admits a right inverse.
\begin{corollary} \label{cor:acyc-retract}
Let $\Ph \maps (A, d, Q) \afib (A',d',Q')$ be an acyclic fibration in $\cAs$. Then there exists
a morphism $\chi \maps (A',d',Q') \to (A,d,Q)$ in $\cAs$ such that $\Ph\circ \chi = \id_{A'}$.
\end{corollary}

\begin{proof}
Using the notation from Prop.\ \ref{prop:min-mod}, recall from Sec.\ \ref{sec:products} that the inclusion $i_{A'} \maps (A',d',Q') \to (A' \tim \ker \sPh, d_\tim, Q_\tim)$ is a morphism in $\cAs$. Let $\chi:= (\Phi,\Psi)^{-1} \cc i_{A'}$. Then $\Phi \cc \chi = \id_{A'}$ as desired.
\end{proof}

Note that a right inverse to an acyclic fibration in $\cAs$ is necessarily a weak equivalence, since weak equivalences in $\cChF$ have the 2 of 3 property.

\subsection{Pullbacks of fibrations and factorizations} \label{sec:pback}
Let  $\Ph=\sPh \maps (A,d,Q) \fib (A'',d'',Q'')$ be a strict fibration in $\cAs$, and $\Theta \maps (A',d',Q') \to (A'',d'',Q'')$ is an arbitrary morphism in $\cAs$. The pullback of the associated diagram  $(A',d') \xto{\Tha^1_1} (A'',d'') \xleftarrow{\sPh} (A,d)$ between tangent complexes in $\cChF$ exists. We recall the explicit description from \cite[\S 4.3.1]{CR}.
Choose a filtration-preserving linear splitting $\si \maps A'' \to A$ of the surjection  
$\sPh \maps A \to A''$ as a morphism in $\cVectF$. Let $h \maps A' \times \ker \sPh \to A' \times A$, and $j \maps A' \times A \to A' \times \ker \sPh$ be the maps in $\cVectF$.
\begin{equation} \label{eq:hj}
h(a',a):=(a', \si \sThe_1(a') + a), \qquad j(a',a):= \bigl(a', a-\si \sThe_1(a') \bigr).
\end{equation}
Then the diagram in $\cChF$ 
\begin{equation} \label{diag:pb-Ch}
\begin{tikzpicture}[descr/.style={fill=white,inner sep=2.5pt},baseline=(current  bounding  box.center)]
\matrix (m) [matrix of math nodes, row sep=2em,column sep=3em,
  ampersand replacement=\&]
  {  
\bigl( \ti{A}, \ti{d} \bigr) \& (A, d) \\
(A',d') \& (A'', d'') \\
}
; 
  \path[->,font=\scriptsize] 
   (m-1-1) edge node[auto] {$\pr \circ h $} (m-1-2)
   (m-1-1) edge node[auto,swap] {$\pr' \circ h$} (m-2-1)
   (m-1-2) edge node[auto] {$\sPh$} (m-2-2)
   (m-2-1) edge node[auto] {$\Theta^1_1$} (m-2-2)
  ;

%begin pullback symbol%
  \begin{scope}[shift=($(m-1-1)!.4!(m-2-2)$)]
  \draw +(-0.25,0) -- +(0,0)  -- +(0,0.25);
  \end{scope}
  %end pullback symbol%
\end{tikzpicture}
\end{equation}
where $\ti{A}= A' \tim \ker \sPh$, and $\ti{d}= j\circ (d' \tim d) \circ h$ is a pullback square. Clearly $\pr' \circ h$ is a fibration in $\cChF$. Furthermore, if $\sPh$ is also a weak equivalence, then it is clear that $\pr' \circ h$ is as well.

The $\cinf$-analog of the next result was proved in \cite[Prop.\ 6.5]{CR}. The same arguments hold in the present context, and so we delay the proof until Appendix \ref{sec:apndx}.  

\begin{proposition} \label{prop:strict-pb}
Let $\Ph=\sPh \maps (A,d,Q) \fib (A'',d'',Q'')$ be a strict fibration/acyclic fibration in $\cAs$ 
and $\Theta \maps (A',d',Q') \to (A'',d'',Q'')$ an arbitrary morphism in $\cAs$. Then the pullback diagram \eqref{diag:pb-Ch} in $\cChF$ lifts through the tangent functor $\ctan \maps \cAs \to \cChF$
to a pullback diagram 
\begin{equation} \label{diag:strict_pullback3}
\begin{tikzpicture}[descr/.style={fill=white,inner sep=2.5pt},baseline=(current  bounding  box.center)]
\matrix (m) [matrix of math nodes, row sep=2em,column sep=3em,
  ampersand replacement=\&]
  {  
(\ti{A},\ti{d},\ti{Q})  \& \bigl (  A, d, Q  ) \\
\bigl   (A', d', Q') \& \bigl (  A'', d'', Q'' ) \\
}; 
  \path[->,font=\scriptsize] 
   (m-1-1) edge node[auto] {$\ppr H$} (m-1-2)
   (m-1-1) edge node[auto,swap] {$\ppr'H$} (m-2-1)
   (m-1-2) edge node[auto] {$\Ph$} (m-2-2)
   (m-2-1) edge node[auto] {$\Theta$} (m-2-2)
  ;

%begin pullback symbol%
  \begin{scope}[shift=($(m-1-1)!.4!(m-2-2)$)]
  \draw +(-0.25,0) -- +(0,0)  -- +(0,0.25);
  \end{scope}
  %end pullback symbol%
\end{tikzpicture}
\end{equation}
in the category in $\cAs$. Moreover, the morphism $\ppr' H \maps( \ti{A}, \ti{d},\ti{Q}) \to (A',d',Q')$ is a fibration/acyclic fibration in $\cAs$.
\end{proposition}

\subsubsection{Factoring $\infty$-morphisms in $\cAs$} \label{sec:factor}
We use Prop.\ \ref{prop:strict-pb} to show that any morphism in $\cAs$ can be factored into a weak equivalence followed by a fibration. Recall from Prop.\ \ref{prop:Ch-pobj} that tensoring 
any object in $\cChF$ with the unital dg algebra  $(N^\ast(\Del^1), \cupp, \unit)$ provides a functorial factorization of the diagonal map into a composition of a weak equivalence with a fibration. Combining this with the functor $ A \ctensor - \maps \dga \to \cAs$  (Sec.\ \ref{sec:Ainf-tensor}), we have for any $\csainf$-algebra $\AdQ{}$ a functorial factorization 
\[
\AdQ{} \xto{\id_A \tensor \unit } (A \tensor N^\ast(\Del^1), d_{N}, Q_{N}) \xto{\id_A \tensor \ro} (A \tim A, d_{\tim}, Q_\tim)
\] 
of the diagonal $\AdQ{} \to \AdQ{} \times \AdQ{}$ in the category $\cAs$.
Above, as in Example \ref{ex:interval}, $\unit \maps \FF \to N^\ast(\Del^1)$ is the unit map, and $\ro \maps N^{\ast}(\Del^1) \to \FF \times \FF$ is the product $\ro=(\ev_0,\ev_1)$ of the evaluation maps on the two vertices of $\Del^1$. In the category $\cChF$, $\unit$ is weak equivalence, $\ro$ is a fibration, and both $\ev_0$ and $\ev_1$ are acyclic fibrations. Hence, in the category $\cAs$, the morphism $\id_A \tensor \unit$ is a strict weak equivalence, $\id_A \tensor \ro$ is 
is a strict fibration, and both $\id_A \tensor \ev_0$ and $\id_A \tensor \ev_1$ are strict acyclic fibrations. 

The main result of this subsection is an instance of Brown's Factorization Lemma \cite{Brown:1973}:
\begin{proposition}\label{prop:Ainf-factor}
Let $\Tha \maps \AdQ{\prime} \to \AdQ{\prime\prime}$ be a morphism in $\cAs$.
\begin{enumerate}
\item There exists a factorization $\Tha = P_{\Tha} \cc \Psi$ in  $\cAs$ such that $P_{\Tha}$ is a fibration, and $\Psi$ is a right inverse to an acyclic fibration (hence a weak equivalence).

\item $\Tha$ is a weak equivalence if and only if the above fibration $P_\Tha$ is an acyclic fibration.
\end{enumerate}
\end{proposition}
\begin{proof}
(1) Use Prop.\ \ref{prop:strict-pb} to take the pullback of the strict acyclic fibration $\id_A \tensor \ev_0$ along $\Tha$. Since $\ev_0 \cc \unit = \id_\FF$, we obtain a commutative diagram 
\[
\begin{tikzpicture}[descr/.style={fill=white,inner sep=2.5pt},baseline=(current  bounding  box.center)]
\matrix (m) [matrix of math nodes, row sep=2em,column sep=3em,
  ampersand replacement=\&]
  {  
(A',d',Q') \& (\ti{A}, \ti{d},\ti{Q})  \& (A'' \tensor \NI, d''_N, Q''_N) \\
\&   (A',d',Q') \&  (A'',d'',Q'') \\
}; 
 \path[->,font=\scriptsize] 
(m-1-1) edge [bend left=20] node[auto]{$(\id_{A''} \tensor \unit)\cc \Tha$} (m-1-3)
(m-1-1) edge [bend right=20] node[auto,swap]{$\id_{A'}$} (m-2-2)
(m-1-2) edge node[auto] {$\ppr H$}  (m-1-3)
  (m-2-2) edge node[auto] {$\Tha$} (m-2-3)
 ;
\path[->>,font=\scriptsize] 
(m-1-2) edge node[auto,swap] {$\ppr'H$} node[sloped, above] {$\sim$} (m-2-2)
(m-1-3) edge node[auto] {$\id_{A''} \tensor \ev_0$} node[sloped, below] {$\sim$} (m-2-3)
;

%begin pullback symbol%
  \begin{scope}[shift=($(m-1-2)!.4!(m-2-3)$)]
  \draw +(-0.25,0) -- +(0,0)  -- +(0,0.25);
  \end{scope}
  %end pullback symbol%
\end{tikzpicture}
\]
Via the universal property, there exists a  morphism $\Psi \maps \AdQ{\prime} \to (\ti{A},\ti{d},\ti{Q})$ into the pullback making the necessary diagrams commute. Proposition \ref{prop:strict-pb} implies that $\ppr'H$ is a weak equivalence, and $\ppr'H \cc \Psi =\id_{A'}$. Hence $\Psi$ is a weak equivalence, as desired, by the 2 of 3 property. 

Next, let $P_{\Tha} \maps (\ti{A},\ti{d},\ti{Q}) \to \AdQ{\prime \prime}$ be the morphism $P_{\Tha}:= 
(\id_{A''} \tensor \ev_1) \cc \ppr H$. Since  $\ev_1 \cc \unit = \id_\FF$, the equality  $P_{\Tha} \cc \Psi = \Tha$ holds. To show that $P_{\Tha}$ is a fibration, it suffices to verify that the linear map $\ctan(P_{\Tha}) \vert_{\cF_n \ti{A}} \maps \cF_n \ti{A} \to \cF_n A^{\prime \prime}$ is degree-wise surjective for all $n \geq 1$. Recall from diagram \eqref{diag:pb-Ch}
that $\ti{A}:= A' \times \ker(\id_{A''} \tensor \ev_0)$. We have $\ctan(P_{\Tha}) = 
(\id_{A''} \tensor \ev_1) \cc (\pr \cc h)$. Here $(\pr \cc h) \maps  A' \times \ker(\id_{A''} \tensor\, \ev_0) \to A'' \tensor \NI$ is the map defined using Eq.\ \ref{eq:hj}: 
\[
(\pr \cc h)(a',x) = \pr(a', \si \sThe_1(a') + x)=\si \sThe_1(a') + x.
\]
where $\si$ is a splitting of the surjection $\id_{A''} \tensor \ev_0$. In degree zero, $N(\Del^1)$ is the 2 dimensional vector space $\FF \vphi_0 \dsum \FF \vphi_1$. Here $\vphi_0$ and $\vphi_1$ are the basis elements defined in Sec.\ \ref{sec:basis}. The evaluation maps satisfy the equalities $\ev_0(\vphi_1) = 0$, and $\ev_1(\vphi_1) = 1$. Therefore, $\cF_nA'' \tensor \FF \vphi_1 \sse 
\cF_n\ker(\id_{A''} \tensor\, \ev_0)$, and so for any $a'' \in \cF_nA''$ we have the equalities
\[
(\id_{A''} \tensor \ev_1) \cc (\pr \cc h)\bigl(0, a'' \tensor \vph_1\bigr) = 
(\id_{A''} \tensor \ev_1) \bigl( a'' \tensor \vph_1\bigr)= a''.
\]
Hence $\ctan(P_{\Tha})\vert_{\cF_n \ti{A}} $ is surjective.

(2) Since $\Psi$ is a weak equivalence, and $\Tha = P_{\Tha} \Psi$, if $\Tha$ is a weak equivalence, then $P_{\Tha}$ is a weak equivalence by the 2 of 3 property.
\end{proof}

\section{A concise proof of the Goldman-Millson Theorem} \label{sec:GM}
We prove an $A_\infty$ analog of the celebrated Goldman-Millson Theorem \cite[Thm.\ 2.4]{GM}. The advantage of the categorical machinery developed in the previous section is that it leads us to a relatively simple proof.

\begin{theorem}\label{thm:GM}
Let $\Tha \maps \AdQ{} \to \AdQ{\prime}$ be a weak equivalence in $\cAs$. Then
\[
\sNb (\Tha) \maps \sNb(A) \to \sNb(A')
\]
is a homotopy equivalence of simplicial sets.
\end{theorem}

We will need one additional lemma which follows directly from Prop.\ \ref{prop:abelian} in the next section.
\begin{lemma}\label{lem:acyc-abelian}
Let $(A,d) \in \cChF$ be an acyclic cochain complex. Then the homotopy groups of the nerve of the 
corresponding abelian $\csainf$-algebra $(A,d,0)$ are trivial: 
\[
\pi_0(\sNb(A)) = 0, \quad \text{and} \quad \pi_{k \geq 1}(\sNb(A),\al) = 0 \quad \forall \al \in \MCAs(A).
\] 
\end{lemma}

\begin{proof}[Proof of Theorem \ref{thm:GM}]
Since $\sNb(A)$ and  $\sNb(A')$ are Kan simplicial sets, we just need to prove that $\sNb(\Tha)$ is a weak homotopy equivalence, i.e.\ that it induces an isomorphism on all homotopy groups. We first consider a special case:  Suppose $\Tha$ is an acyclic fibration. Let $(\ker \Tha^1_1,d)$ denote the abelian $\csainf$-sub-algebra of $\AdQ{}$ corresponding to the kernel of $\ctan(\Tha)$. Then Prop.\ \ref{prop:min-mod} implies that there exists a morphism $\Psi \maps \AdQ{} \to (\ker \Tha^1_1,d)$ which, in turn, yields an isomorphism
\[
(\Tha, \Psi) \maps (A, d, Q) \xto{\iso} (A' \tim \ker \Tha^1_1, d_\tim, Q_\tim),
\]
% \todo[fancyline]{this all comes from Prop. 5.2}
via Cor.\ \ref{cor:min-mod}. In order to show $\sNb(\Tha)$ is a weak homotopy equivalence, we deduce from diagram \eqref{diag:acyc-strict} that it suffices to show that $\sNb(\Pr_A') \maps 
\sNb(A' \tim \ker \Tha^1_1) \to \sNb(A')$ is a weak homotopy equivalence. Recall from Prop.\ \ref{prop:Nprod} that the nerve functor preserves products. Hence, $\sNb(\Pr_A')$ is a weak equivalence if and only if the projection
\[
\pr_{\sNb(A')} \maps \sNb(A') \tim \sNb(\ker \Tha^1_1) \to \sNb(A')
\] 
induces an isomorphism on all homotopy groups at all basepoints $(\al',\al)$. The latter statement immediately follows from 
Lemma \ref{lem:acyc-abelian}, since we have isomorphisms
\[
\begin{split}
\pi_{0}\bigl( \sNb(A') \tim \sNb(\ker \Tha^1_1) \bigr) &\cong 
\pi_{0}\bigl( \sNb(A') \bigr) \tim \pi_0 \bigl(\sNb(\ker \Tha^1_1) \bigr) \cong \pi_{0}\bigl( \sNb(A') \bigr) \\
\pi_{k \geq 1}\bigl( \sNb(A') \tim \sNb(\ker \Tha^1_1),(\al',\al) \bigr) &\cong
\pi_{k \geq 1}\bigl( \sNb(A'), \al' \bigr) \tim \pi_{k \geq 1} \bigl(\sNb(\ker \Tha^1_1), \al \bigr)\\
&  \cong \pi_{k \geq 1}\bigl( \sNb(A'), \al' \bigr).
\end{split}
\]

This completes the  theorem for the case when $\Tha$ is an acyclic fibration in $\cAs$. Now we consider the general case. Let $\Tha \maps \AdQ{} \to \AdQ{\prime}$ be a weak equivalence. By Prop.\ \ref{prop:Ainf-factor}, there exists a factorization $\Tha = P_{\Tha} \cc \Psi$ in $\cAs$ such that $P_\Tha$ is an acyclic fibration, and $\Psi$ is a right inverse to an acyclic fibration, which we denote as $\Upsilon$ in the following diagram:
\[
\begin{tikzdiag}{1}{3}
{
\AdQ{}\& (\ti{A}, \ti{d},\ti{Q} ) \& \AdQ{\prime} \\
};

\path[->>,font=\scriptsize]
(m-1-2) edge node[auto,swap] {$\Upsilon$} node[sloped,below] {$\sim$} (m-1-1)
(m-1-2) edge node[auto] {$P_{\Tha}$} node[sloped,below] {$\sim$} (m-1-3)
;
\path[->,font=\scriptsize]
(m-1-1) edge[bend left =40]  node[auto] {$\Psi$} node[sloped,below] {$\sim$} (m-1-2)
;
\end{tikzdiag}
\]
As was verified above, $\sNb(\Upsilon)$ and $\sNb(P_{\Tha})$ are  weak homotopy equivalences since $\Upsilon$ and $P_{\Tha}$ are  acyclic fibrations in $\cAs$. The equality $\sNb(\Upsilon) \cc \sNb(\Psi) = \id_{\sNb(A)}$ implies that $\sNb(\Psi)$ is a weak homotopy equivalence by the 2 of 3 property. Hence, we conclude $\sNb(\Tha) = \sNb(P_{\Tha}) \cc \sNb(\Psi)$ is a weak homotopy equivalence as well.
\end{proof}

\begin{remark}\label{rmk:GM}
One could continue on this path and show that the category $\cAs$ forms a category of fibrant objects (CFO) for a homotopy theory \cite{Brown:1973}, in complete analogy with the category of $\cinf$-algebras in characteristic zero \cite[Thm.\ 6.2]{CR}. All that remains is to verify that Prop.\ \ref{prop:strict-pb} generalizes to pullbacks of arbitrary (that is, not strict) fibrations and acyclic fibrations. The proof used to verify the analogous $\cinf$ statement \cite[Cor.\ 6.6]{CR} would work equally well here. After this, one could proceed to strengthen Thm.\ \ref{thm:GM} by showing that the nerve functor is an ``exact functor'' between $\cAs$ and the category of Kan simplical sets, in analogy with the simplicial Maurer-Cartan functor for $\cinf$-algebras \cite[Thm.\ 7.4]{CR}. One just needs to verify that $\sNb(-)$ preserves pullbacks of fibrations.
We leave these details as an exercise for a suitably motivated reader.
\end{remark}

\begin{remark}  \label{rmk:curved}
The main results of de Kleijn and Wierstra \cite{DW} hold for curved $\csainf$-algebras, which are more general than the flat $\csainf$-algebras considered in this paper. It is reasonable to suspect that a modification of our proof of Thm.\ \ref{thm:GM} could provide a Goldman-Millson Theorem for curved algebras. The main issue is finding a convenient presentation for the homotopy theory of curved $\csainf$-algebras in positive characteristic. Indeed, it follows directly from the definition \cite[Def.\ 3.1]{DW} that a curved $\csainf$-algebra does not, in general, have an underlying cochain complex. The homotopy theory developed in Sec.\ \ref{sec:cat-props} for flat $\csainf$-algebras, therefore, does not automatically extend to the curved case. 

One promising approach, which circumvents the above issue, is to instead prove a version of the Goldman-Millson Theorem for the  ``mixed-curved'' $\csainf$-algebras featured in the recent work of Calaque, Campos, and Nuiten \cite[Example 2.72]{CCN}. Mixed curved algebras have an underlying cochain complex, and the notion of weak equivalence between such algebras coincides with our notion of weak equivalence between flat $\csainf$-algebras. The upshot is the characterization of the homotopy theory for curved algebras in terms of mixed-curved algebras via a pullback square of simplicial categories \cite[Thm.\ C; Cor.\ 2.64]{CCN}. However, the authors      
of \cite{CCN} work over a field of characteristic zero, and indeed, their construction of the simplicial category of curved $\csainf$-algebras explicitly uses the de Rham-Sullivan algebra.
A modification of this construction would be needed in order to verify that the relevant results of \cite{CCN} extend to positive characteristic. This would take us beyond the scope of the present paper, and so we leave it for future work. 
\end{remark}

\section{The homotopy groups of the nerve} \label{sec:hmtpy}
In this section, we characterize the homotopy groups of $\sNb(A)$ in terms of the cohomology of  tangent cochain complex of $\AdQ{}$. Throughout, we tacitly use the notation from Section \ref{sec:simp-note} for simplicial sets and the cosimplicial simplicial set $\Del^\bul$.

\subsection{The abelian case} \label{sec:ab}
Following the terminology frequently used in the context of $L_\infty$-algebras, we say an $\csainf$-algebra $\AdQ{}$ is \df{abelian} if $\AdQ{}$ is equal to its tangent cochain complex $(A,d)$, or equivalently $\sQ_{n} =0$ for all $n >1$. In this case, the homotopy groups of $\sNb(A)$ are easy to characterize. Recall that one side of the Dold-Kan correspondence \cite{Dold,Kan} is given by the functor  $K_\bul \maps \ChainF \xto{\simeq} \sVectF$, where 
\[
K_\bul(C):= \Hom_{\ChainF}(N_\ast(\Del^\bul), C).
\]
Furthermore, there are isomorphisms of abelian groups, natural in $C$:
\begin{equation} \label{eq:DK}
\pi_n(K_\bul(C),0) \cong H_n(C) \quad \forall n \geq 0.
\end{equation}

\begin{proposition} \label{prop:abelian}
Let $\AdQ{}=(A,d)$ be an abelian  $\sainf$-algebra. There is an isomorphism of abelian groups
\[
\pi_0(\sNb(A)) \cong H^0(A), \qquad \pi_n(\sNb(A),\al) \cong H^{-n}(A) \quad \forall n \geq 1,
\] 
for each vertex $\al \in \sN_0(A)=\MCAs(A)$.
\end{proposition}
\begin{proof}
By definition of the nerve, $\sN_n(A) = \MCAs(A \tensor N^\ast(\Del^n)) = Z^0(A \tensor N^\ast(\Del^n))$, since $A$ is abelian.
Because the cochain complex $N^\ast(\Del^n)$ is concentrated in degrees $0,\ldots,n$, we have $\sN_n(A) \sse \bigoplus_{k=0}^n A^{-k} \tensor N^k(\Del^n)$. Consider the simplicial vector space $K_\bul(C)$, where $C \in \ChainF$ is the chain complex whose underlying graded vector space is $C_k := A^{-k}$ for all $k \geq 0$, and whose differential is given by $d \maps A^{-k} \to A^{-k +1}$.   
By taking the linear dual of the chain complex $N_\ast(\Del^n)$, we may write $K_n(C)$
as degree zero cocycles of the tensor product of the cochain complex $N^\ast(\Del^n)$ with the truncation $\tau_{\leq 0}A$ of $A$. As a result, we obtain a natural isomorphism of simplicial vector spaces 
\[
K_\bul(C) \cong Z^0 \bigl(A \tensor N^\ast(\Del^\bul)  \bigr). 
\]
Hence, $\pi_0(\sNb(A)) \cong H^0(A)$ and $\pi_n(\sNb(A),0) \cong H^{-n}(A)$ for all $n \geq 1$ by \eqref{eq:DK}. Since $\sNb(A) \cong K_\bul(C)$ is a simplicial abelian group, translation by any vertex $\al \in \MCAs(A)$ gives an automorphism $\sNb(A) \cong \sNb(A)$. Therefore, for  all $n \geq 1$, $\pi_{n}(\sNb(A),\al) \cong \pi_{n}(\sNb(A),0) \cong  H^{-n}(A)$.
\end{proof}

\subsection{$\pi_{n \geq 2}$ for the general case} \label{sec:gen}
Fix an $\csainf$-algebra $(A,d_A,Q)$. Here we will characterize the homotopy groups $\pi_{n \geq 2}(\sNb(A),0)$. We treat $\pi_{1}(\sNb(A),0)$ separately in the next section, for the sake of exposition. Our approach % for determining the homotopy groups
follows that of Berglund's \cite{Berglund} for complete filtered $L_\infty$-algebras in characteristic zero.
By applying the below results to the twisted $\csainf$-algebra $(A,{d_A}^\al,Q^\al)$, 
the analogous characterization of the groups $\pi_{n \geq 1}(\sNb(A),\al)$ is obtained, thanks to Lemma \ref{lem:shift}.  

Our conventions for simplicial homotopy groups follow \cite[Sec.\ 5.2, 5.3]{Wu}. For $n \geq 1$, let $\sphZ_n(\sNb(A),0)$ denote the set of spherical $n$-simplicies:
\[
\sphZ_n(\sNb(A),0):= \{ \be \in \MCAs(A \tensor N^\ast(\Del^n)) \st d_i \be =0 \quad \forall i=0,\ldots,n\}, 
\] 
and let $Z^{-n}(A)$ denote, as usual, the degree $-n$ cocycles of the tangent cochain complex $(A,d_A)$.
For any  $a \in A^{-n}$, the curvature equation (Sec.\ \ref{sec:MC}) along with Eq.\ \ref{eq:top} implies that $\curv_{\Ass}(a \tensor \vph_{[n]}) = d_Aa \tensor \vph_{[n]}$. Hence,
\[
a \tensor \vph_{[n]} \in \MCAs(A \tensor N^\ast(\Del^n)) \qquad \forall a \in Z^{-n}(A).
\]
Furthermore, Eq.\ \ref{eq:top} also implies that the equalities $d_i(a \tensor \vph_{[n]})= a \tensor d_i\vph_{[n]} =0$ hold for all $i=0,\ldots, n$. Therefore, we have a well-defined function  $\chi_n \maps Z^{-n}(A) \to \sphZ_n(\sNb(A),0)$
\[
\chi_n(a):= a \tensor \vph_{[n]}.
\]  
\begin{lemma}\label{lem:chi-surj}
For each $n\geq 1$, the function $\chi_n \maps Z^{-n}(A) \to \sphZ_n(\sNb(A),0)$ is surjective.
\end{lemma}
\begin{proof}
Any element $\be \in \MCAs(A \tensor N^\ast(\Del^n))$ may be written as $\be = \sum_{k=0}^n \sum_{ \si \in \cB(n,k)} \be_{k,\si} \tensor \vph_\si$, for some $\be_{k,\si} \in A^{-k}$. If, in addition, $\be \in \sphZ_n(\sNb(A),0)$, then we have
$\sum_{k=0}^n \sum_{ \si \in \cB(n,k)} \be_{k,\si} \tensor d_i \vph_\si  =0$ for all $i=0,\ldots,n$.
To prove the lemma, it follows from the definition of the function $\chi_n$ that 
it suffices to show that $\be_{k,\si}=0$ for all $\si \in \cB(n,k)$ whenever $k \neq n$. If $\si \in \cB(n,k\neq n)$ is such a $k$-simplex in $\Del^n$, then there exists an $i \in \{0,\ldots,n\}$ such that $i \notin \si$. Hence, there is a non-degenerate $k$-simplex $\ta \in \cB(n-1,k)$ such that $d^i\ta =\si$. From Eq.\ \ref{eq:face}, it then follows that $\be_{k,\si} \tensor \vph_{\ta}$ will appear as a summand on the left hand side of the equation $d_i \be =0$. Therefore, $\be_{k,\si}=0$ by linear independence of the dual basis.
\end{proof}

Next, recall that for $n \geq 1$, we may express the $n$th homotopy group as the quotient $\pi_n(\sNb(A),0) =\sphZ_n(\sNb(A),0)/ \sim$, where $\be \sim \be'$ if there exists $ \eta \in \sN_{n+1}(A)$ such that $d_0\eta = \be$, $d_1 \eta = \be'$, and $d_j \eta =0$ for all $j > 1$. 
The group operation on $\pi_{n \geq 2}(\sNb(A),0)$ is given by $[\al] + [\be] = [d_1 \om]$, where $\om \in  \sN_{n+1}(A)$ is any $n+1$ simplex satisfying: 
\begin{equation} \label{eq:grp-op}
d_0 \om = \be, \quad d_2 \om = \al, \quad \text{$d_j \om = 0$ for all $j > 2$.}
\end{equation}

\begin{theorem} \label{thm:pin}
Denote by $\ba{\chi}_n \maps Z^{-n}(A) \to \pi_n(\sNb(A),0)$ the function $\ba{\chi}_n(a):= [a \tensor \vph_{[n]}]$. Then
\begin{enumerate}
\item $\ba{\chi}_n$ is a homomorphism of abelian groups for all $n \geq 2$.
\item The homomorphism $\ba{\chi}_n$ induces an isomorphism
\[
H^{-n}(A) \xto{\cong} \pi_n(\sNb(A),0) \quad \forall n \geq 2.
\] 
\end{enumerate}
\end{theorem}
\begin{proof}
(1) Let $a,b \in Z^{-n}(A)$. It suffices to exhibit an $n+1$ simplex $\om \in \MCAs\bigl (A \tensor N^\ast(\Del^{n+1}) \bigr)$ as in \eqref{eq:grp-op} such that $[d_1 \om] =  [(a+b) \tensor \vph_{[n]}]$. Let $\om \in A^{-n} \tensor N^n(\Del^{n+1})$ be the element
\begin{equation} \label{eq:chi-bij-1}
\om := b \tensor \vph_{d^0[n]} + (a+b) \tensor \vph_{d^1[n]} + a \tensor \vph_{d^2[n]}, 
\end{equation}
where, for example, $d^1[n]$ is the $n$-simplex $(0 < 2 < 3 <\cdots < n <n+1)$  in $\Del^{n+1}$.
By using the formulas for the face maps \eqref{eq:face}, it is clear that $\om$ satisfies the conditions given in \eqref{eq:grp-op}. It remains to verify that $\om$ is Maurer-Cartan. Since $n \geq 2$, the cup product of any two basis elements appearing on the right hand side of Eq.\ \ref{eq:chi-bij-1} vanish, for degree reasons. Therefore, we have
\[
\curv_{\Ass}(\om) = (d_A \tensor \id)(\om) + (\id \tensor \del)(\om)=(\id \tensor \del)(\om),
\]  
where the last equality above follows from the fact that $a$ and $b$ are cocycles. 
Formula \eqref{eq:diff} for the differential on normalized cochains gives us the equalities
$\del \vph_{d^j[n]} = (-1)^{n+1 +j} \vph_{[n+1]}$, for $j=0,1,2$. Hence, $(\id \tensor \del)(\om) = (-1)^{n+1} \bigl( b - (a+b) +a \bigr)\tensor \vph_{[n+1]}=0$, and therefore $\om \in \MCAs\bigl(A \tensor N^\ast(\Del^{n+1}) \bigr)$. 

(2) By Lemma \ref{lem:chi-surj}, it suffices to show that $\ker  \ba{\chi}_n$ is equal to the subgroup of coboundaries. Suppose $a = d_Aa' \in Z^{-n}(A)$. By definition of the equivalence relation on spherical elements, in order to show $\ba{\chi}_n(a) = 0$, it suffices to construct an $n+1$ simplex $\eta \in 
\MCAs(A \tensor N^\ast(\Del^{n+1}))$ such $d_0 \eta = a \tensor \vph_{[n]}$ and $d_j\eta =0$ for $j >0$. Define $\eta:= d_Aa' \tensor \vph_{d^0[n]} + a' \tensor \vph_{[n+1]}$. 
Then $\eta$ clearly satisfies the required boundary conditions. Since $n \geq 2$, a computation identical to the one involving $\om$ in the above proof of part (1) shows that $\curv_{\Ass}(\eta) = 0$.

Conversely, suppose $a \in Z^{-n}(A)$ and $\ba{\chi}_n(a) = [0]$. Then there exists an
$n+1$ simplex $\eta \in \MCAs(A \tensor N^\ast(\Del^{n+1}))$ such $d_0 \eta = a \tensor \vph_{[n]}$ and $d_j\eta =0$ for $j =1, \ldots, n+1$. Write $\eta$ as $\eta = \sum_{k=0}^{n+1} \sum_{ \si \in \cB(n+1,k)} \eta_{k,\si} \tensor \vph_\si$, with $\eta_{k,\si} \tensor \vph_\si \in A^{-k} \tensor N^{k}(\Del^{n+1})$.  
The constraint on $d_0 \eta$ implies that the only non-trivial summand of $\eta$ with bidegree $(-n,n)$
is $a \tensor \vph_{d^0[n]}$, and also that $\sum_{k\neq n} \sum_{ \si \in \cB(n+1,k)} \eta_{k,\si} \tensor d_0\vph_\si=0$.
This equation along with the remaining constraints on $d_{j}\eta$ imply that 
\[
\sum^{n-1}_{k=0} \sum_{ \si \in \cB(n+1,k)} \eta_{k,\si} \tensor \vph_\si=0,
\]
via the argument given in the proof of Lemma \ref{lem:chi-surj}. Hence, we conclude that $\eta$ is of the form 
\[
\eta = a \tensor \vph_{d^0[n]} + a' \tensor \vph_{[n+1]} \in A^{-n} \tensor N^{n}(\Del^{n+1}) \dsum
A^{-n-1} \tensor N^{n+1}(\Del^{n+1}).
\] 
Finally, since $n \geq 2$, the Maurer-Cartan equation $\curv_{\Ass}(\eta)=0$ implies that $\eta$ is a cocycle, and therefore $a=d_A a'$ is a coboundary. 
\end{proof}

\subsection{$\pi_{1}$ for the general case} \label{sec:gen-pi1}
The non-abelian group operation on $\pi_1(\sNb(A),0)$ is $[\al] \cdot [\be] := [d_1\om]$,
where $\om \in \sN_2(A)$ is any $2$-simplex satisfying $d_0 \om = \be$, $d_2 \om = \al$. To characterize this in terms of the complex $(A,d_A)$, we need a new operation on the degree $-1$ cohomology. The origin of this new operation is classical. It appears, for example, in the theory of Jacobson radicals for non-unital noncommutative rings \cite[p.\ 67]{Lam}.   
For the sake of motivation, consider a local unital ring $S$. Let  $S^\times$ denote its group of units, and $\mm$ its maximal ideal. If $x \in \mm$, then we have $1+x \in S^\times$. The inverse of $1+x$ is necessarily of the form $1+y$ with $y \in \mm$. As an element of the non-unital ring $\mm$, 
$y$ is determined by the fact that $x +y + xy =0$ if and only if $(1+x)(1+y) =1$. These elementary observations lead to the following definition.    
\begin{definition}[Sec.\ 1 \cite{Bergman}] \label{def:qi}
Let $\ba{R}$ be a non-unital associative $\FF$ algebra. The \df{quasi-multiplication} 
$-\qm- \maps \ba{R} \times \ba{R} \to \ba{R}$  is the associative binary operation
\[
a \qm b:= a + b + ab, \quad \forall a,b \in \ba{R}.
\]
An element $a \in \ba{R}$ is \df{quasi-invertible} if there exists $\qinv{a} \in \ba{R}$ such that
$a \qm \qinv{a} = \qinv{a} \qm a =0$.
\end{definition}
The quasi-multiplication induces a group structure on the subset $\ba{R}^\qm \sse \ba{R}$ of quasi-invertible elements\footnote{The term ``quasi-invertible'' is adopted from \cite{Bergman}. It should not be confused with the notion of  ``quasi-invertible'' 1-simplicies in the sense of \cite[Sec.\ 7]{BG}.}.
\begin{remark}\label{rmk:unitize}
If $R$ is an augmented unital associative $\FF$ algebra with augmentation ideal $\ba{R}$, then there is an isomorphism of groups $\ba{R}^{\qm} \cong \{ u \in R^\times \st 1-u \in \ba{R}\}$. 
% Let $(-)^u \maps \nuAlg \to \Alg$ denote the 
% left adjoint to the forgetful functor from unital algebras to non-unital algebras.
% An element $a$ in a non-unital algebra $R$ is quasi-invertible if and only if $1_{R^u}+ a$ is invertible in $R^u$.
\end{remark}

\begin{lemma}\label{lem:quasi-pronil}
Every element of a pronilpotent non-unital associative algebra is quasi-invertible.
\end{lemma}
\begin{proof}
The quasi-inverse of any element can be constructed explicitly via the usual formula for  inverting a power series.
\end{proof}
% \begin{example}\label{ex:pronil}
Fix a $\csainf$ algebra $(A,d_A,Q)$. The equation $Q^2=0$ implies a compatibility between the map $\sQ_2$ and the differential $d$, so that the binary operation on $H^{-1}(A)$:
\[
\ba{a} \ba{b}:= \ov{\sQ_2(a,b)}
\]
is well-defined and associative. The filtration on $A$ induces a complete filtration on $H^{-1}(A)$ compatible with the above multiplication. Hence, $H^{-1}(A)$ is a pronilpotent non-unital associative algebra, and $H^{-1}(A)^{\qm} = H^{-1}(A)$ by Lemma \ref{lem:quasi-pronil}. 

% \end{example}

It turns out that the group structure on $H^{-1}(A)$ given by the quasi-multiplication $\qm$ coincides with the concatenation of loops in the fundamental group of $\sNb(A)$.  
\begin{theorem}\label{thm:pi1}
The function $\ba{\chi}_1 \maps Z^{-1}(A) \to \pi_1(\sNb(A),0)$, defined as $\ba{\chi}_1(a):= [a \tensor \vph_{[1]}]$ induces a group isomorphism
\[
H^{-1}(A)^\qm \xto{\cong} \pi_1(\sNb(A),0).
\]
\end{theorem}

To prove the theorem, we will use a technical lemma.
\begin{lemma}\label{lem:MC2}
A degree $0$ element $\al \in A \tensor N^\ast(\Del^2)$ of the form
\begin{equation} \label{eq:MC2}
\al = \sum_{j=0}^2 w_j \tensor \vph_{d^{\, j}[1]} + u \tensor \vph_{[2]} \in A^{-1} \tensor N^{1}(\Del^2) 
~ \dsum ~
A^{-2} \tensor N^{2}(\Del^2)
\end{equation}
is Maurer-Cartan if and only if $w_0,w_1,w_2 \in Z^{-1}(A)$, and
\[
d_A u - w_0 + w_1 - w_2 - \sQ_2(w_2,w_0)=0.
\]
\end{lemma}
\begin{proof}
Let $Q_N$ denote the $\csainf$-structure on $A \tensor N^\ast(\Del^2)$ induced from  $(A,d_A,Q)$
as in Sec.\ \ref{sec:Ainf-tensor}.
Via Eq. \ref{eq:Ainf-tensor1}, we see that the hypothesis on $\al$ implies that $(Q_N)^1_k(\al^{\tensor k})$ will involve $k-1$ cup products of positive degree elements in $N^{\ast}(\Del^2)$. Therefore
\[
{\curv_{\Ass}}_N(\al) = \sum_{k=1}^\infty (Q_N)^1_k(\al^{\tensor k}) = (d_A \tensor \id + \id \tensor \del)(\al) + (Q_N)^1_2(\al \tensor \al).
\]   
The formula \eqref{eq:diff} for the differential $\del$ implies that  
$\del \vph_{d^{\, j}[1]} = (-1)^{j} \vph_{[2]}$. Formula \eqref{eq:cup}  implies that
$\vph_{d^{\, i}[1]} \cupp \vph_{d^{\, j}[1]} = \vphi_{[2]}$ if $i=2$, and $j=0$, and that all other cup products vanish. Hence, after applying the Koszul rule for signs, we see that the Maurer-Cartan equation ${\curv_{\Ass}}_N(\al) =0$ simplifies to
\[
\sum_{j=0}^2 d_A w_j \tensor \vph_{d^{\, j}[1]} + \Bigl( d_Au - \sum_{j=0}^2 (-1)^j w_j - \sQ_2(w_2,w_0) \Bigr)\tensor \vph_{[2]} =0.
\]
\end{proof}

\newcommand{\bfchi}{\boldsymbol{\chi}}

\begin{proof}[Proof of Theorem \ref{thm:pi1}]
Let us first show that $\ba{\chi}_1$ descends to an injective function $\bfchi_1 \maps  H^{-1}(A) \to \pi_1(\sNb(A),0)$. It will then follow from Lemma \ref{lem:chi-surj} that $\bfchi_1$ is a bijection between sets.  

Suppose $a, b$ are $1$-cocycles such that $a - b =d_Au$. Let $\al:= a \tensor \vph_{d^{0}[1]} + b \tensor \vph_{d^{1}[1]}  + u \tensor \vph_{[2]} \in \bigl(A \tensor N^\ast(\Del^2) \bigl)^0$. Then $\al \in \sN_2(A)$ by Lemma \ref{lem:MC2}, and
Eq.\ \ref{eq:face} implies that $d_0 \al = a \tensor \vph_{[1]}$,  $d_1 \al = b\tensor \vph_{[1]}$, and
$d_2 \al = 0$. Therefore, $\ba{\chi}_1(a)=\ba{\chi}_1(b)$, and so the function $\bfchi_1$ is well-defined. 
For the converse direction, suppose  $\ba{\chi}_1(a)=\ba{\chi}_1(b)$. Then there is an $\eta \in \sN_2(A)$ witnessing the equivalence
of spherical elements $a \tensor \vph_{[1]} \sim  b \tensor \vph_{[1]}$. Write $\eta$ as $\eta = \ti{\al} + \al$, where $\al$ is a degree $0$ element of the form \eqref{eq:MC2} above, and $\ti{\al} \in A^0 \tensor N^0(\Del^2)$. For each basis element $\vph_{i} \in N^{0}(\Del^2)$ either $d_0\vph_{i}$ or $d_1\vph_{i}$ is non-trivial. Therefore, since $d_0\eta$ and $d_1\eta$ are in bidegree $(-1,1)$, it follows that $\ti{\al} =0$. We then conclude from Lemma \ref{lem:MC2} that $\eta = a \tensor \vph_{d^{0}[1]} + b \tensor \vph_{d^{1}[1]}  + u \tensor \vph_{[2]}$ for some $u \in A^{-2}$ such that $d_Au = a-b$. Hence, $\bfchi_1$ is injective.    

It remains to show that $\bfchi_1$ is a group homomorphism. Let $a, b \in Z^{-1}(A)$. Then the equation $Q^2 = 0$ implies that $\sQ_2(a,b) \in Z^{-1}(A)$ as well. Let $\om \in A \tensor N^{\ast}(\Del^2)$ be the degree $0$ element 
\[
\om:= b \tensor \vph_{d^{0}[1]} + \bigl(a+b + \sQ_2(a,b) \bigr) \tensor \vph_{d^{1}[1]} 
+ a \tensor \vph_{d^{2}[1]}. 
\]
Then $\om \in \sN_2(A)$ by Lemma \ref{lem:MC2}. Since $d_0\om = b \tensor \vph_{[1]}$, $d_2\om = a \tensor \vph_{[1]}$, the definition of the group operation $\pi_1(\sNb(A),0)$ implies that
\[
[a \tensor \vph_{[1]}] \cdot  [b \tensor \vph_{[1]}] = [d_1\om] = [a+b + \sQ_2(a,b)]. 
\]
Hence, $\bfchi_1(\ba{a} \qm \ba{b}) = \bfchi_1(\ba{a}) \cdot \bfchi_1(\ba{b})$.
\end{proof}
\release{\bfchi}

\newcommand{\GMC}[1]{\MC_{\mathrm{dgAlg}}(#1)/\sim_{\mathrm{gauge}}}
\newcommand{\MCA}{\MC_{\mathrm{dgAlg}}}

\subsection{$\pi_0$ for dg algebras}
Suppose $(C,d_C,\mu)$ is a complete filtered non-unital dg associative algebra over $\FF$. Let $(A,d,Q)$ be the corresponding $\csainf$-algebra, as in Example \ref{ex:dga}, with $A:=\bs^{-1}C$. Then $A^{-1}=C^0$ is
a pronilpotent associative algebra, and so by Lemma \ref{lem:quasi-pronil} we can consider $C^0$ as the group of quasi-invertible elements $G:=(C^0,\qm)$. Elements of $G$ act on the Maurer-Cartan\footnote{It is standard to consider Maurer-Cartan elements of dg associative algebras as degree 1 elements.} set $\MCA(C):=\{x \in C^1 \st d_C x + \mu(x,x)=0\}$ via the so-called ``gauge action''
\begin{equation} \label{eq:gauge}
g \cdot x : = x - d_C^x(g) - d^x_C(g)\qinv{g}.
\end{equation}
for all $x \in \MCA(C)$ and $g \in C^0$. Above $d_C^x(g):= d_C(x) + \mu(x,g) -\mu(g,x)$ is the differential twisted by $x$, and $\qinv{g}$ is the quasi-inverse of $g$. This action is well known for unital dg algebras. See, for example, the exposition in \cite[Sec.\ 5.6]{WE} for the non-unital case.
Denote by $\GMC{C}$ the set of Maurer-Cartan elements modulo the gauge action.

On the other hand, $\MCA(C)$ is the vertex set of the nerve $\sNb(A)$, and we may consider the set of equivalence classes $\pi_0 \sNb(A)$ induced by simplicial homotopy. Given $x, y \in \MCA(C)$, we write $x \simh y$ if $\bs^{-1}x$ and $\bs^{-1}y$ are connected by a 1-simplex in $\sNb(A)$.
% In this case, Maurer-Cartan elements $x,y \in \MC(C)$ are equivalent if there exists $\ga \in \MC \bigl(\bs^{-1}C \tensor N^{\ast}(\Del^1) \bigr)$ such that $(\id \tensor \ev_0)\ga = \bs^{-1}x$ and $(\id \tensor \ev_1)\ga = \bs^{-1}y$. Here $\ev_0$ and $\ev_1$ are the evaluation homomorphisms from  Example \ref{ex:interval}.
\begin{lemma}\label{lem:dga-pi0}
Let $(A,d,Q)$ be the corresponding $\csainf$-algebra of
$(C,d_C,\mu)$ as above.  A degree $0$ element $\ga \in A \tensor N^\ast(\Del^1)$ of the form
\begin{equation} \label{eq:MC2}
\begin{split}
\ga = (\bs^{-1}y \tensor \vph_0 + \bs^{-1}x \tensor \vph_1)  + \bs^{-1}g \tensor \vph_{[1]}&  \\ 
&\in A^{0} \tensor N^{0}(\Del^1)~ \dsum ~ A^{-1} \tensor N^{1}(\Del^1)
\end{split}
\end{equation}
is a 1-simplex in $\sNb(A)$ if and only if $x,y \in \MCA(C)$ and 
\[
 y =  x - \mu(y,g) +  \mu(g,x)  - d_Cg.
\]
\end{lemma}
\begin{proof}
Let $d_N$ and $Q_N$ denote the induced differential and $\csainf$-structure, respectively, on $A \tensor N^\ast(\Del^1)$. Consider the curvature 
${\curv_{\Ass}}_N(\ga)= d_N \ga + (Q_N)^1_2(\ga,\ga)$  of the degree zero element $\ga$. The non-trivial values of the differential on $N^\ast(\Del^1)$ are $\del \vphi_0 = \vphi_{[1]}$, and $\del \vphi_1 = -\vphi_{[1]}$. Therefore:
\[
\begin{split}
d_N \ga &= \bs^{-1} d_Cy \tensor \vphi_0 + \bs^{-1} d_Cx \tensor \vphi_1  
+ \bs^{-1} d_C g \tensor \vphi_{[1]} + \bs^{-1} y \tensor \del \vphi_0 + \bs^{-1}x \tensor \del \vphi_1\\
&= \bs^{-1} d_Cy \tensor \vphi_0 + \bs^{-1} d_Cx \tensor \vphi_1  + \bs^{-1}\bigl( d_C g + y -x \bigr) \tensor \vphi_{[1]}.
\end{split}
\]
Recall that for $a,b \in C$, the definition of the codifferential $Q$ implies that $Q^1_2(\bs^{-1}a,\bs^{-1}b) = (-1)^{\deg{a} -1} \bs^{-1} \mu(a,b)$. The non-trivial products in $N^\ast(\Del^1)$ are: $\vph_0 \cupp \vph_{[1]} = \vph_{[1]} \cupp \vph_{1} = \vph_{[1]}$, $\vph_0 \cupp \vph_0  = \vph_0$, and $\vph_1 \cupp \vph_1  = \vph_1$. Therefore, we obtain the equalities 
\[
\begin{split}
(Q_N)^1_2(\ga,\ga) &= (Q_N)^1_2 ( \bs^{-1} y \tensor \vphi_0, \bs^{-1} y \tensor \vphi_0) + 
(Q_N)^1_2 ( \bs^{-1} y \tensor \vphi_0, \bs^{-1} g \tensor \vphi_{[1]}) \\
& \quad + (Q_N)^1_2 ( \bs^{-1} x \tensor \vphi_1, \bs^{-1} x \tensor \vphi_1) +
(Q_N)^1_2 ( \bs^{-1} g \tensor \vphi_{[1]}, \bs^{-1} x \tensor \vphi_1)\\
&= \bs^{-1} \mu(y,y) \tensor \vphi_0 + \bs^{-1} \mu(x,x) \tensor \vphi_1 + 
\bs^{-1} \bigl( \mu(y,g) - \mu(g,x) \bigr) \tensor \vphi_{[1]}.
\end{split}
\]
The statement then follows, since the cochains $\vph_0, \vph_1, \vph_{[1]}$ are linearly independent. 
\end{proof}

\begin{proposition} \label{prop:gauge}
Let $(A,d,Q)$ be the corresponding $\csainf$-algebra of a complete filtered non-unital dg associative algebra.
$(C,d_C,\mu)$. Then Maurer Cartan elements $x, y \in \MCA(C)$ are gauge equivalent if and only if $x \simh y$, and so there is a bijection of sets 
\[
\GMC{C} ~ \cong ~ \pi_0(\sNb(A)).
\]
\end{proposition}
\begin{proof}
We show here that if $x \simh y$, then there exists a $g \in C^0$ such that $y=g \cdot x$. The converse direction will follow from essentially the same computation.
Let us write $wz$ for the product $\mu(w,z)$ of elements $w,z \in C$.
Then, from Lemma \ref{lem:dga-pi0}, there is an element $g \in C^0$ such that
\begin{equation} \label{eq:gauge1}
y=  x - yg + gx - d_C g,
\end{equation}
Multiplying both sides of the equality on the right by the quasi-inverse $\qinv{g}$ of $g$ gives
$y\qinv{g} =  x\qinv{g} - yg\qinv{g} + gx\qinv{g} - (d_C g)\qinv{g}$. We then add this to the original equality  Eq.\ \ref{eq:gauge1} to obtain
\begin{equation}
\begin{split}
y + y\qinv{g}  &=  x +  x\qinv{g} + gx + gx\qinv{g} - yg - yg\qinv{g}
- d_C g   - (d_C g ) \qinv{g}\\
\end{split} 
\end{equation}
Since $y(g + \qinv{g} + g\qinv{g}) = y(g \qm \qinv{g}) =0$, we have
\[
y =  x +  x\qinv{g} + gx + gx\qinv{g} - d_C g   - (d_C g ) \qinv{g}.
\]  
A similar application of the identity $x(g \qm \qinv{g}) =0$ gives a further simplification:
\[
y =  x  - (d_C g + xg -gx)   -  (d_C g +xg - gx ) \qinv{g}.
\]
Hence, $y= g \cdot x$.
\end{proof}

\subsection{Application to deformations of group representations} \label{sec:def-app}
The formal deformation theory of a group representation is a classical example \cite{DF:1974, GS:1987, WE} of a deformation problem governed by the Maurer-Cartan set of a dg associative algebra.
Let $G$ be a group and fix a representation of $G$, i.e.\ a morphism of associative algebras  $\rho \maps \F[G] \to \End_{\F}(V)$. Pre/post-composition by $\rho$ gives $\End_{\F}(V)$ the structure of a bimodule over the group ring $\F[G]$, and so we may consider the Hochschild complex $C^\ast(G,V)^\ro$. This is the cochain complex whose degree $n$ component is the vector space $\Hom_\FF \bigl(\F[G]^{\tensor n}, \End_{\FF}(V) \bigr)$, and the differential is built using the multiplication on $\F[G]$ and the representation $\rho$. The associative multiplication in $\End_{\F}(V)$ induces an associative multiplication  on cochains. Hence, $C:=\bigl(C^\ast(G,V)^\ro,d, \mu, \bigr)$ is a dg associative algebra. In what follows,  $A:=(A,d,Q)$ denotes $\sainf$-algebra  corresponding to $C$.

Given a complete local Noetherian $\FF$-algebra $(R,\mm)$, the differential on $A$, along with the products on $C$ and $\mm$, extend to the completed tensor product $C \ctensor \mm:= \plim_n C \tensor_\FF (\mm/\mm^n)$ in the obvious way. The result is a complete non-unital dg algebra $C_R:=(C \ctensor \mm, d_{R}, \mu_{R} \bigr)$.
The Maurer-Cartan set $\MCA(C_R)$ encodes the deformations of $\rho$ over $R$, or equivalently lifts of $\rho$ to a representation on  $V \tensor_\FF R$. The corresponding ``coarse moduli space'' 
of such deformations is the set of orbits $\GMC{C_R}$. Proposition \ref{prop:gauge} gives a bijection
\begin{equation} \label{eq:moduli}
\GMC{C_R} ~ \cong ~ \pi_0 \bigl( \sNb(A_R) \bigr)  
\end{equation}
where $A_R:=(A \ctensor R,d_R,Q_R)$ in the notation of Sec.\ \ref{sec:Ainf-tensor}.

\newcommand{\tf}{\mathrm{tran}}
Taking a step back, let $H:=H^\ast(A,d)$ denote the tangent cohomology of $A$ as a graded vector space, i.e.\ the shifted Hochschild cohomology $\bs^{-1}HH^\ast \bigl(\F[G],\End_{\F}(V) \bigr)$. 
Via the homotopy transfer theorem \cite{Kad}, there exists an $\sainf$-algebra structure $Q_\tf$ on the complex $(H,0)$, and a $\sainf$-quasi-isomophism
\begin{equation} \label{eq:trans}
\Ph \maps (A,d,Q) \xto{\sim} (H,0,Q_\tf).
\end{equation}
The $\sainf$-structure $Q_\tf$ extends the associative product induced by the product on the Hochschild complex. As a result,  $(H,0,Q_\tf)$ completely encodes the dg algebra structure on Hochschild cochains. The upshot is that the underlying graded vector space of
$(H,0,Q_\tf)$ is ``smaller'' and hence  more amenable to computations than the dg algebra $C$. 

Moreover, the transferred structure on the cohomology retains all of the information about the formal deformation theory of $\rho$. Indeed, given $(R,\mm)$ as above, we take the completed tensor product of the quasi-isomorphism \eqref{eq:trans} with $\mm$ to obtain a weak equivalence in $\cAs$:       
\[
\Ph_R \maps A_R \weq H_R,
\]
where $H_R$ is $H \ctensor \mm$ equipped with the $\csainf$-structure induced by $Q_\tf$. In particular, applying the Goldman-Millson theorem (Thm. \ref{thm:GM}) to this weak equivalence, along with the bijection \eqref{eq:moduli} we obtain
\begin{equation} \label{eq:gauge-trans}
\GMC{C_R} \, \xto{\cong} \,  \pi_0 \bigl( \sNb(H_R) \bigr).  
\end{equation}
This identifies gauge equivalence classes of deformations of $\rho$ over $R$ with 
simplicial homotopy classes of Maurer-Cartan elements in $H_R$. 

\begin{remark}\label{rmk:WE}
In his work on deformations of Galois representations, a bijection similar to \eqref{eq:gauge-trans} was obtained by C.\ Wang-Erickson \cite[Thm.\ 6.2.3]{WE} via a more explicit computational approach. A direct calculation shows that two Maurer-Cartan elements of an $\csainf$-algebra $\AdQ{}$ are `strictly gauge equivalent'' in the sense of 
Wang-Erickson \cite[Def.\ 5.6.5]{WE} if and only if they represent the same class in 
$\pi_0\bigl(\sNb(A) \bigr)$.
\end{remark}

\section{Comparisons in characteristic zero} \label{sec:char0}
In this section, we set $\FF=\kk$ to be a field of characteristic zero. We denote by $\sOm{n}$, the unital commutative dg assocative algebra (cdga) of polynomial de Rham forms on the $n$-simplex. The collection $\sOmb:= \{\sOm{n} \}_{n \geq 0}$ forms a simplicial unital cdga. Let 
\begin{equation} \label{eq:Om-unit}
\unit_\Om \maps \kk \to \sOmb
\end{equation}
denote the unit, as a morphism between simplicial cdga. For each $n \geq 0$, the map $\unit_{\Om n}$
is a quasi-isomorphism of cochain complexes. Hence, $\sOmb$ satisfies the ``Poincar\'{e} Lemma'', in the sense of Rmk.\ \ref{rmk:poincare}. The simplicial cdga $\sOmb$ provides another way to construct a simplicial set from any $\csainf$-algebra $\AdQ{}$ in characteristic $0$. In contrast with the nerve, $\sNb(A)$, this simplicial set can be infinite-dimensional, even if $A$ is finite type.
\begin{definition} \label{def:sMCA}
The \df{Maurer-Cartan simplicial set} of $\AdQ{} \in \cAs$ is the simplicial set $\sMCAb(A)$, where
\[
\sMCA{n}(A):= \MCAs \bigl(A \ctensor \sOm{n}\bigr) \qquad \forall n\geq 0.
\]
\end{definition} 
Since the above assignment is built by composing the functor $A \ctensor - \maps \dga \to \cAs$ 
with the Maurer-Cartan functor $\MCAs(-)$ from Sec.\ \ref{sec:MC}, it induces a functor $\sMCAb \maps \cAs \to \sSet$.

\subsection{Complete shifted $L_\infty$-algebra} \label{sec:Lie}
A $\cinf$-algebra, or \df{complete shifted $L_\infty$-algebra} $(L,d,Q)$ is a  complete cochain complex $(L,d) \in \Ch(\kk)$ equipped with a degree $1$ codifferential $Q$ on the corresponding cofree cocommutative coalgebra $\bar{S}(L)$ which is compatible with the filtration on $L$ in the obvious way.
Similarly, a filtered $\infty$-morphism  between $\cinf$-algebras is a dg coalgebra map which is compatible with the filtrations. For the full details, we refer the reader to \cite[Sec.\ 5]{CR}, and the references cited there. We denote by $\cLie$ the category of $\cinf$-algebras and filtered $\infty$-morphisms.

Given $\LdQ{} \in \cLie$, recall that the \df{curvature}  $\curv_{\Lie} \maps L^0 \to L^{1}$ is the function which assigns to a degree zero element $x \in L^0 = \cF_1L^0$ the infinite series 
\begin{equation} \label{eq:curvLie}
\curv_\Lie(x)=d x + \sum_{m \geq 2} \frac{1}{m!}
\sQ_m(\underbrace{x\cdots x}_{\text{$m$ {\rm times}}}). 
\end{equation}

The Maurer-Cartan set of $\LdQ{}$ is the set of degree $0$ elements with vanishing curvature:
$\MC_\Lie(L):= \{ x \in L^0 \st \curv_{\Lie}(x)=0\}$. It extends to a functor 
\begin{equation} \label{eq:MCLie}
\MC_{\Lie}(-) \maps \cLie \to \Set. 
\end{equation}
The value of $\MC_{\Lie}$ on a morphism $\Ph \maps \LdQ{} \to \LdQ{}$ is given by 
$\MC_{\Lie}(\Phi)(x) = \sum_{n \geq 1} \frac{1}{n!}\Ph^1_n(x^n)$. (See \cite[Sec.\ 5.4]{CR} for more details.)\, Note the similarity with the definition of the  functor $\MCAs(-) \maps \cAs \to \Set$, in Eq.\ \ref{eq:phistar}. In particular, if $\Ph=\sPh \maps \LdQ{} \to \LdQ{\prime}$ is a strict morphism, then $\MC_{\Lie}(\Phi)(x) = \sPh(x)$, just like for $\cainf$-algebras.
As explained, for example in \cite[Sec.\ 5.6]{CR}, 
taking the MC elements of the completed tensor product of $\LdQ{}$ with $\sOmb$ functorially produces a Kan simplicial set. The \df{Maurer--Cartan simplicial set} $\sMCLb(L)$
of an $\cinf$-algebra $(L,d,Q)$ is the Kan simplicial set with $n$-simplices 
\[
\sMCL{n}(L):= \MC_{\Lie}\bigl(L \ctensor \sOm{n} \bigr).
\]
We denote by $\sMCLb \maps \cLie \to \Kan$ the corresponding functor.

\newcommand{\SQ}{Q^{\, \sym}} 
\subsection{Commutator brackets} \label{sec:com}
The commutator bracket gives any associative algebra the structure of a Lie algebra. T.\ Lada and M.\ Markl \cite[Thm.\ 3.1]{LM}  showed that this generalizes to an analogous relationship between $\sainf$-algebras and $\sinf$-algebras over a field of characteristic zero. Their construction easily extends to the complete filtered case. Given $\AdQ{} \in \cAs$, denote by $(\cL(A),d,\SQ)$ the $\cinf$-algebra whose underlying cochain complex is $(\cL(A),d):=(A,d)$, and whose codifferential on $\bar{S}(A)$ is given by
\begin{equation} \label{eq:sym}
(\SQ)^1_n(x_1 x_2\cdots x_n):= \sum_{\si \in S_n} \varepsilon(\si) Q^1_n(x_{\si(1)} \tensor x_{\si(2)} \tensor \cdots \tensor x_{\si(n)}).
\end{equation}
By following the proof of \cite[Thm.\ 3.1]{LM}, it is clear that the assignment 
$\AdQ{} \mapsto (\cL(A),d,\SQ)$ is functorial with respect to {\it strict} filtration preserving morphisms. Indeed, if $\Ph = \sPh \maps \AdQ{} \to \AdQ{\prime}$ is strict, then the equality
$Q^{\prime 1}_n \cc\Ph^n_n = \Ph^1_1 \cc Q^1_n$ implies that $(Q^{\prime \sym})^1_n \cc \Ph^n_n = \Ph^1_1\cc (\SQ)^1_n$. Hence, the corresponding functor between the subcategories containing only strict morphisms
\begin{equation} \label{eq:L}
\cL \maps \strcAs \to \strcLie
\end{equation}
is just the identity on hom-sets.

A direct comparison between the formulae for the curvature functions  
$\curv_{\Ass} \maps A^0 \to A^1$ and $\curv_\Lie \maps \cL(A)^0 \to \cL(A)^1$ shows that they have the same zero locus, i.e.\ $\MCAs(A) = \MC_\Lie \bigl( \cL(A) \bigr)$. Furthermore, our discussion above concerning the value of the functor \eqref{eq:MCLie} on strict morphisms implies that there is a natural isomorphism -- in fact, an equality -- of functors
\begin{equation} \label{eq:MCnat}
\MCAs(-) = \MC_{\Lie}\cc \cL(-) \maps \strcAs \to \Set.
\end{equation}
The analogous statement holds for the simplicial Maurer-Cartan sets, as well.
\begin{proposition} \label{prop:sMC-equal}
The functor $\sMCAb \maps \strcAs \to \Kan$ is equal to the composition\\
$\sMCLb \cc \cL(-) \maps \strcAs \to \Kan$.
\end{proposition}
\begin{proof}
Given $\AdQ{} \in \cAs$ and a cdga $B \in \cdga$, recall from Sec.\ \ref{sec:Ainf-tensor} that the completed tensor product $(A \ctensor B, d_B, Q_B)$ inherits a $\csainf$ structure from $\AdQ{}$. Since the multiplication on $B$ is graded commutative, it follows from Eq.\ \ref{eq:sym} that we have an equality of $\cinf$-algebra\footnote{From here on,  for the sake of readability, we will sometimes omit the (co)differentials from our notation for $\csainf$ and $\cinf$-algebras.}
$\cL\bigl(A \ctensor B \bigr) = \cL(A) \ctensor B$. This upgrades to an equality  of functors
\[
\cL(A \ctensor -) = \cL(A) \ctensor - \maps \cdga \to \strcLie.
\]
% since $A \ctensor - \maps \cdga \to \cAs$ and
% $\cL(A) \ctensor - \maps \cdga \to \cLie$ send dg algebra morphisms to strict $\infty$-morphisms.
Combining this with \eqref{eq:MCnat}, we obtain a sequence of equalities
\[
\begin{split}
\sMCA{n}(A) = \MCAs(A \ctensor \sOm{n}) &= \MC_{\Lie}\bigl( \cL( A \ctensor \sOm{n}) \bigr)\\
&= \MC_{\Lie}\bigl( \cL( A) \ctensor \sOm{n} \bigr) = \sMCL{n}\bigl(\cL(A) \bigr).
\end{split}
\]
It then follows from the definition of the functors $\cL(-)$, $\MCAs(-)$, and $\MC_{\Lie}(-)$, that the above equalities
are natural with respect to strict morphisms between $\cainf$-algebras.
\end{proof}
\begin{corollary}\label{cor:sMC-equal}
If $ \Ph \maps \AdQ{} \weq \AdQ{\prime}$ is a strict weak equivalence in $\strcAs$ then
$\sMCAb(\Ph) \maps \sMCAb(A) \to \sMCAb(A')$ is a homotopy equivalence of simplicial sets.
\end{corollary}
\begin{proof}
The Goldman-Millson Theorem for $\cinf$-algebras \cite[Thm.\ 1.1]{DR} implies that\\
$\sMCLb(\Ph) \maps \sMCLb(\cL(A)) \to \sMCLb(\cL(A'))$ is a homotopy equivalence. Hence, the corollary follows from Prop.\ \ref{prop:sMC-equal}.
\end{proof}

\newcommand{\cS}[1]{\mathcal{S}_{#1}}
\newcommand{\cSb}{\cS{\bl}}
\subsection{Comparing $\sNb(A)$ with $\sMCLb(\cL(A))$} \label{sec:N-vs-sMC}
N.\ de Kleijn and F.\ Wierstra conjectured in \cite{DW} that the 
nerve $\sNb(A)$ of a $\cainf$-algebra $\AdQ{}$ is homotopy equivalent to the simplicial Maurer-Cartan set $\sMCLb(\cL(A))$ of its commutator $\cinf$-algebra $(\cL(A),d,\SQ)$.
In this section, we prove this statement.

In fact, we will prove something a bit stronger. To establish the setting, 
let $\sdga$ denote the category of simplicial unital dg associative algebras. The main examples are the normalized cochain algebra $N^\ast(\Del^\bl)$ from Sec.\ \ref{sec:norm-form}, and the polynomial de Rham forms $\sOmb$ described above. The tensor product in $\sdga$ extends, level-wise, to a symmetric monoidal structure on $\sdga$. In particular, we denote by $(\Om \tensor N)_\bl$ the simplicial dg algebra with $(\Om \tensor N)_n:= \sOm{n} \tensor N^\ast(\Del^n)$. % The braiding gives an isomorphism $\tau_{\Om,N} \maps (\Om \tensor N)_\bl \xto{\cong} (N \tensor \Om)_\bl$
% of simplicial dg algebras.

It is clear that the construction of the simplicial Maurer-Cartan functor $\sMCAb(-)= \MCAs(-\ctensor \, \sOmb)$ generalizes in the following way: For any $B_\bl \in \sdga$ let
\begin{equation} \label{eq:MC-sdga}
\cSb^B \maps \cAs \to \sSet
\end{equation} 
be the simplicial set valued functor $\cSb^B(A):= \MCAs(A \ctensor B_\bl)$. We now state the main result.
\begin{theorem}\label{thm:N-vs-MC}
Let $\AdQ{}$ be a $\cainf$-algebra. There are weak homotopy equivalences of simplicial sets
\begin{equation} \label{diag:N-vs-MC}
\begin{tikzdiag}{2}{2}
{
\sNb(A)\& \cSb^{\Om \tensor N}(A)\& \sMCLb \bigl(\cL(A) \bigr) \\
};
\path[->,font=\scriptsize]
(m-1-1) edge node[auto] {$f_A$}node[auto,below]{$\sim$} (m-1-2)
(m-1-3) edge node[auto,swap] {$g_A$}node[auto,below]{$\sim$} (m-1-2)
;
\end{tikzdiag}
\end{equation}
which are natural with respect to strict $\infty$-morphisms: Given  $\Tha = \Tha^1_1 \maps \AdQ{} \to \AdQ{\prime}$ in $\strcAs$, the following diagram commutes
\begin{equation} \label{diag:nat}
\begin{tikzdiag}{2}{2} 
{
\sNb(A)\& \cSb^{\Om \tensor N}(A)\& \sMCLb \bigl(\cL(A) \bigr) \\
\sNb(A')\& \cSb^{\Om \tensor N}(A')\& \sMCLb \bigl(\cL(A') \bigr) \\
};
\path[->,font=\scriptsize]
(m-1-1) edge node[auto] {$f_A$}node[auto,below]{$\sim$} (m-1-2)
(m-1-3) edge node[auto,swap] {$g_A$}node[auto,below]{$\sim$} (m-1-2)
(m-2-1) edge node[auto] {$f_{A'}$}node[auto,below]{$\sim$} (m-2-2)
(m-2-3) edge node[auto,swap] {$g_{A'}$}node[auto,below]{$\sim$} (m-2-2)
(m-1-1) edge node[auto,swap] {$\sNb(\Tha)$} (m-2-1)
(m-1-2) edge node[auto] {$\cSb^{\Om \tensor N}(\Tha)$} (m-2-2)
(m-1-3) edge node[auto] {$\sMCLb(\cL(\Tha))$} (m-2-3)
;
\end{tikzdiag}
\end{equation}

\end{theorem}

To prove the theorem, we will first need to collect some facts concerning the homotopy theory of bisimplicial sets. Our primary reference is \cite[Ch.\ IV]{GJ}.
\subsubsection{Bisimplicial objects} \label{sec:bisimp}
Recall that given a category $\catC$, we denote by $\sC$ the category of simplicial objects in $\catC$. Analogously, let $\ssC$ be the category of \df{bisimplicial objects} $X = \{X_{mn} \}_{m,n\geq 0}$ in $\catC$, i.e.\ functors $X \maps \Del^{\op}  \times \Del^{\op} \to \catC$, with $X(m,n):=X_{mn}$. We use the convention that for $X \in \ssC$, $X_{mn}$ denotes the object of simplicies in ``horizontal dimension'' $m$ and ``vertical dimension'' $n$. At the same time, we will also treat $\ssC$ as the category of simplicial objects in $\sC$. If $X$ is a bisimplicial object, then denote by $X(-)_\bl \maps \Del^{\op} \to \sC$ the functor which sends the ordinal $[m]$ to the $m$th vertical simplicial object $X(m)_\bl:=X_{m \bul} \in \sC$.  

Let $\diag \maps \Del^{\op} \to \Del^{\op} \times \Del^{\op}$ and $p_2 \maps \Del^{\op} \times \Del^{\op} \to \Del^{\op}$ be the diagonal functor $\diag([m]):=([m],[m])$,  and the projection functor  $p_2([m],[n]):=[n]$, respectively. Pre-composition with $\diag$ and $p_2$ induces functors between the  categories $\ssC$ and $\sC$:
\[
\diag^{\sharp} \maps \ssC \to \sC, \quad \diag^{\sharp}(X)_{m}:= X_{mm} \quad \forall m \geq 0,
\]
and
\[
p_2^{\sharp} \maps \sC \to \ssC, \quad  p_2^{\sharp}(S)_{mn}:=S_n \quad \forall m,n \geq 0.
\]
% \begin{align*}
% &\diag^{\ast} \maps \ssC \to \sC &  &\diag^{\ast}(X)_{n}:= X_{nn}\\
% &p_2^{\ast} \maps \sC \to \ssC & & p_2^{\ast}(A)_{mn}:=A_n \quad \forall m,n \geq 0
% % &\diag^{\ast} \maps \ssC \to \sC &  &p_2^{\ast} \maps \sC \to \ssC\\
% % &\diag^{\ast}(X)_{n}:= X_{nn} &   &p_2^{\ast}(A)_{mn}:=A_n \quad \forall m,n \geq 0 
% \end{align*}
If $S_{\bul} \in \sC$, then the functor
$p_2^\sharp(S)(-)_\bl   \maps \Del^{op} \to \sC$ is just the constant functor $[m] \mapsto S_\bl$. Also, note that we have an equality of functors $\diag^{\sharp} \! \cc \,  p_2^{\sharp} = \id_{\sC}.$

\begin{definition} \label{def:bi-weq}
A morphism of bisimplicial sets  $f \maps X \to Y$ in $\ssSet$ is a \df{point-wise weak equivalence}
if, for all $m \geq 0$, $f_{m \bul} \maps X(m)_{\bul} \to Y(m)_{\bul}$ is a weak homotopy equivalence in $\sSet$.
\end{definition}

\begin{proposition}[Prop.\ IV.1.7 \cite{GJ}] \label{prop:bi-weq}
If $f \maps X \to Y$ is a point-wise weak equivalence in $\ssSet$ then
$\diag^\sharp(f) \maps \diag^{\sharp}(X) \to \diag^{\sharp}(Y)$ is a weak homotopy equivalence in $\sSet$.
\end{proposition}

\subsubsection{Proof of Theorem \ref{thm:N-vs-MC}}
\paragraph{The weak equivalence $f_A \maps \sNb(A) \to \cSb^{\Om \tensor N}(A)$\\}
We will first construct the weak equivalence $f_A$ in diagram \eqref{diag:N-vs-MC}. Let $m \geq 0$. As discussed above, the simplicial unit map \eqref{eq:Om-unit} for $\sOmb$ gives a quasi-isomophism $\unit_{\Om_m} \maps \kk \weq \sOm{m}$ of dg algebras. 
By applying the functor $A \ctensor - \maps \dga \to \strcAs$ to this quasi-isomorphism,
we obtain, via Lemma \ref{lem:ctensor}, a strict weak equivalence of $\cainf$-algebras
\[
\Phi_A(m):= \id_A \ctensor \unit_{\Om_m}  \maps  \AdQ{} \weq (A \ctensor \sOm{m}, d_{\Om_m}, Q_{\Om_m}).
\]
Therefore, for all $m \geq 0$, the Goldman-Millson theorem (Thm.\ \ref{thm:GM}) implies that the induced map between the nerves
is a homotopy equivalence:
\begin{equation} \label{eq:f1}
\sNb(\Ph_A(m)) \maps \sNb(A) \xto{\sim} \sNb(A \ctensor \sOm{m}).
\end{equation}
For each $m,n \geq 0$, let $X_{mn}$ denote the set 
\[
X_{mn}:= \sN_{n}(A \ctensor \sOm{m})= \MCAs\bigl( (A \ctensor \sOm{m}) \ctensor N^\ast(\Del^n) \bigr).
\]
The functoriality of the composition  $\sN_\bl \cc (A \ctensor -) \maps \dga \to \sSet$ implies 
that the collection  $\{X_{mn}\}$ forms a bisimplicial set $X \in \ssSet$. Recall from  
Lemma \ref{lem:ainf-tensor} that we can naturally reparenthesize the tensor products in $\strcAs$:
$(A \ctensor \sOm{m}) \ctensor N^\ast(\Del^n) \cong 
A \ctensor (\sOm{m} \tensor N^\ast(\Del^n))$. As a result, we obtain a natural isomorphism of simplicial sets $\diag^\sharp (X) \cong \cSb^{\Om \tensor N}(A)$.

Next, given a morphism between ordinals $\al \maps [\el] \to [k]$, let $\al^\sharp \maps \sOm{k} \to \sOm{\el}$ denote the corresponding morphism of dg algebras. Since the unit of $\sOmb$ is a simplicial map, the equality $\al^\sharp \cc \unit_{\Om_k} = \unit_{\Om_\el}$ holds, and therefore
\[
\sNb( \id_A \ctensor \al^\sharp) \cc  \sNb(\Ph_A(k)) =\sNb(\Ph(\el)).
\]   
Hence, the collection of simplicial maps $\{\sNb(\Ph_A(m))\}_{m \geq 0}$ assemble together to form a morphism of bisimplicial sets $\vph_A \maps p_2^{\sharp}(\sNb(A)) \to X$,  
where
\[
\vph_{A_{mn}}:=\sNb(\Ph_A(m))_n \maps \sN_n(A) \to X_{mn}.
\]  
From the weak equivalences \eqref{eq:f1}, we see that $\vph_A$ is point-wise weak equivalence of bisimplicial sets. Let $f_A \maps 
\sNb(A) \to \cSb^{\Om \tensor N}(A)$ be the image of $\vphi_A$ under the diagonal functor $\diag^\sharp \maps \ssSet \to \sSet$. Then Prop.\ \ref{prop:bi-weq} implies that $f_A$ is a weak equivalence of simplicial sets.

Finally, let  $\Tha:=\Tha^1_1 \maps \AdQ{} \to \AdQ{\prime}$ be a strict $\infty$-morphism. For each $m \geq 0$, the equality $\Tha \ctensor \id_{\sOm{m}} \, \cc \, \Ph_A(m) = \Ph_{A'}(m) \cc \Tha$ holds as morphisms in $\strcAs$ between $\AdQ{}$ and $(A' \ctensor \sOm{m}, d'_{\Om_m}, Q'_{\Om_m})$. All of the constructions used in building the maps $f_A$ and $f_{A'}$ involve either natural isomorphisms between tensor products or applying functors to the morphisms $\Ph_A(m)$ and $\Ph_{A'}(m)$. Hence, the diagram
\begin{equation} \label{diag:nat1}
\begin{tikzdiag}{2}{2} 
{
\sNb(A)\& \cSb^{\Om \tensor N}(A)\\
\sNb(A')\& \cSb^{\Om \tensor N}(A')\\
};
\path[->,font=\scriptsize]
(m-1-1) edge node[auto] {$f_A$}node[auto,below]{$\sim$} (m-1-2)
(m-2-1) edge node[auto] {$f_{A'}$}node[auto,below]{$\sim$} (m-2-2)
(m-1-1) edge node[auto,swap] {$\sNb(\Tha)$} (m-2-1)
(m-1-2) edge node[auto] {$\cSb^{\Om \tensor N}(\Tha)$} (m-2-2)
;
\end{tikzdiag}
\end{equation}
 commutes. This proves the first half of the naturality statement \eqref{diag:nat}.

\paragraph{The weak equivalence $g_A \maps \sMCLb(\cL(A)) \to \cSb^{\Om \tensor N}(A)$\\}
Proposition \ref{prop:sMC-equal} gives a natural identification $\sMCAb(A) = \sMCLb(\cL(A))$.
The tensor product of simplicial dg algebras is a symmetric monoidal structure on the category $\sdga$. Hence, the braiding provides a natural isomorphism $\cSb^{\Om \tensor N} \simeq \cSb^{N \tensor \Om}$ of functors from $\strcAs$ to $\sSet$. 
Therefore, it suffices to construct a map of simplicial sets
\[
\ti{g}_A \maps \sMCAb(A) \to \cSb^{N \tensor \Om}(A)   
\]
and verify that it is a weak homotopy equivalence. We proceed as in the construction of $f_A$ above, swapping the roles played by the simplicial dg algebras $N^\ast(\Del^\bl)$ and $\sOmb$.
The simplicial unit  $\unit_{N} \maps \kk \weq N^\ast(\Del^\bl)$ provides, for each $m \geq 0$, 
a strict weak equivalence of $\cainf$-algebras
\[
\Psi_A(m):= \id_A \ctensor \unit_{N_m}  \maps  \AdQ{} \weq (A \tensor N^\ast(\Del^m), d_{N}, Q_{N}).
\]
Corollary \ref{cor:sMC-equal} implies that $\sMCAb(\Psi_A(m)) \maps \sMCAb(A) \to \sMCAb(A \tensor N^\ast(\Del^m))$ is homotopy equivalence. Let $Y \in \ssSet$ denote the bisimplicial set
\[
Y_{mn}:= \sMCA{n}(A \tensor N^\ast(\Del^m)) = \MCAs\bigl( (A \tensor N^\ast(\Del^m)) \ctensor \sOm{n}     \bigr)
\]
By Lemma \ref{lem:ainf-tensor}, there is a strict natural isomorphism of $\cainf$-algebras
$\bigl(A \tensor N^\ast(\Del^m) \bigr) \ctensor \sOm{n} \cong
A \ctensor \bigl(N^\ast(\Del^m) \tensor \sOm{n}\bigr)$. Hence, $\diag^\sharp (Y) \cong \cSb^{N \tensor \Om}(A)$. The collection of simplicial maps $\{\sMCAb(\Psi_A(m))\}$ defines a point-wise weak equivalence of bisimplicial sets $\psi_A \maps p^{\sharp}_2(\sMCAb(A)) \to Y$, where
$\psi_{A_{mn}}:= \sMCA{n}(\Psi_A(m))$. Applying the diagonal functor gives a weak homotopy equivalence
\[
\ti{g}_A:= \diag^\sharp(\psi_A) \maps \sMCAb(A) \weq \cSb^{N \tensor \Om}(A), 
\]
thanks to Prop.\ \ref{prop:bi-weq}. 

Finally, if $\Tha:=\Tha^1_1 \maps \AdQ{} \to \AdQ{\prime}$ is a strict $\infty$-morphism, then the argument demonstrating the commutativity of diagram \eqref{diag:nat1} also implies ({\it mutatis mutandis}) that the diagram
\begin{equation} 
\begin{tikzdiag}{2}{2} 
{
\sMCAb(A)\& \cSb^{N \tensor \Om}(A)\\
\sMCAb(A')\& \cSb^{N \tensor \Om}(A')\\
};
\path[->,font=\scriptsize]
(m-1-1) edge node[auto] {$\ti{g}_A$}node[auto,below]{$\sim$} (m-1-2)
(m-2-1) edge node[auto] {$\ti{g}_{A'}$}node[auto,below]{$\sim$} (m-2-2)
(m-1-1) edge node[auto,swap] {$\sMCAb(\Tha)$} (m-2-1)
(m-1-2) edge node[auto] {$\cSb^{N \tensor \Om}(\Tha)$} (m-2-2)
;
\end{tikzdiag}
\end{equation}
commutes. Combining this with the natural isomorphisms $\sMCLb(\cL(A)) \cong \sMCAb(A)$, and 
$\cSb^{\Om \tensor N}(A) \cong \cSb^{N \tensor \Om}(A)$ verifies the second half of the naturality statement \eqref{diag:nat}.

This completes the proof of Thm.\ \ref{thm:N-vs-MC}. \hfill \qed

\appendix
\addtocontents{toc}{\protect\setcounter{tocdepth}{0}}

\section{Appendix: Proof of Proposition \ref{prop:strict-pb}} \label{sec:apndx}
Let $\Ph=\sPh \maps (A,d,Q) \fib (A'',d'',Q'')$ be a strict fibration in $\cAs$ 
and $\Theta \maps (A',d',Q') \to (A'',d'',Q'')$ an arbitrary morphism in $\cAs$. 
Here we prove the existence of a pullback diagram in $\cAs$ containing $\Theta$ and $\Ph$ which lifts the diagram \eqref{diag:pb-Ch} in $\cChF$ through the tangent functor $\ctan \maps \cAs \to \cChF$.

We proceed exactly as in the case for $\cinf$-algebras over a field of characteristic zero \cite[Sec.\ 6.1.1]{CR}\footnote{The proof in {\it loc.\ cit}.\ is the complete filtered analog of a construction given by B.\ Vallette \cite{Vallette:2014} for pullbacks of $\infty$-morphisms between homotopy algebras.}.
First, recall that the pullback diagram \eqref{diag:pb-Ch} of the tangent maps in $\cChF$ is  
\[
\begin{tikzpicture}[descr/.style={fill=white,inner sep=2.5pt},baseline=(current  bounding  box.center)]
\matrix (m) [matrix of math nodes, row sep=2em,column sep=3em,
  ampersand replacement=\&]
  {  
\bigl( \ti{A}, \ti{d} \bigr) \& (A, d) \\
(A',d') \& (A'', d'') \\
}
; 
  \path[->,font=\scriptsize] 
   (m-1-1) edge node[auto] {$\pr \circ h $} (m-1-2)
   (m-1-1) edge node[auto,swap] {$\pr'$} (m-2-1)
   (m-1-2) edge node[auto] {$\sPh$} (m-2-2)
   (m-2-1) edge node[auto] {$\Theta^1_1$} (m-2-2)
  ;

%begin pullback symbol%
  \begin{scope}[shift=($(m-1-1)!.4!(m-2-2)$)]
  \draw +(-0.25,0) -- +(0,0)  -- +(0,0.25);
  \end{scope}
  %end pullback symbol%
\end{tikzpicture}
\]
with $\ti{A}= A' \tim \ker \sPh$, and $\ti{d}= j\circ (d' \tim d) \circ h$. Here, as in Eq.\ \ref{eq:hj}, $h$ and $j$ are the maps $h(a',a):=(a', \si \sThe_1(a') + a)$ and
$j(a',a):= \bigl(a', a-\si \sThe_1(a') \bigr)$ in $\cVectF$. The map $\si$ is a chosen splitting of $\sPh \maps A \to A''$ as a morphism in $\cVect$.

Denote elements in $A' \tim A$ as vectors  $\bb{a}:=(a',a)$. Let  $H^1_k \maps \T^k(A' \tim A) \to  A' \tim  A$, and $J^1_k \maps \T^k( A' \tim A) \to  A' \tim A$ be the linear maps
\begin{equation} \label{eq:H}
\begin{split}
&H^1_1(\bb{a}):= h(\bb{a}), \quad  J^1_1(\bb{a}):= j(\bb{a}), \\
&H^1_{k}(\bb{a}_1,\ldots,\bb{a}_k):= \bigl(0,  \si\sThe_{k}(a'_1,\ldots,a'_k) \bigr), \\
&J^1_{k}(\bb{a}_1,\ldots,\bb{a}_k):= \bigl(0, - \si \sThe_{k}(a'_1,\ldots,a'_k) \bigr).
\end{split}
\end{equation}
A direct calculation shows that $H  J = J  H = \id_{\T(A' \tim A)}$. Since $\T(\ti{A}) \sse \T(A' \dsum A)$, we can define 
a degree 1 codifferential $\ti{Q} \maps \T(\ti{A}) \to \T(\ti{A})$ via
$\ti{Q}:= J \cc Q_{\tim} \cc  H \vert_{\T(\ti{A})}$,
% \begin{equation} \label{eq:deltilde}
% \ti{Q}:= J \cc Q_{\tim} \cc  H \vert_{\T(\ti{A})}
% \end{equation}
where $Q_{\tim}$ is the $\sainf$-structure  on $A' \tim A$ induced by $Q'$ and $Q$. By construction, $\ti{Q}^1_1 = \ti{d}$,  $\ti{Q}^2 =0$, and a direct calculation shows that $ \im \ti{Q} \sse \T(\ti{A})$. Since $H$, $J$, and $\ti{Q}$ are all constructed from filtration preserving maps,
$(\ti{A}, \ti{d}, \ti{Q})$ is an object in $\cAs$, and
$H \vert_{\T(\ti{A})} \maps (\ti{A}, \ti{d},\ti{Q}) \to (A' \tim  A, d_{\tim},Q_\tim) $
is a morphism in $\cAs$.

We now observe that
\begin{equation} \label{diag:strict_pullback3}
\begin{tikzpicture}[descr/.style={fill=white,inner sep=2.5pt},baseline=(current  bounding  box.center)]
\matrix (m) [matrix of math nodes, row sep=2em,column sep=3em,
  ampersand replacement=\&]
  {  
(\ti{A},\ti{d},\ti{Q})  \& \bigl (  A, d, Q  ) \\
\bigl   (A', d', Q') \& \bigl (  A'', d'', Q'' ) \\
}; 
  \path[->,font=\scriptsize] 
   (m-1-1) edge node[auto] {$\ppr H$} (m-1-2)
   (m-1-1) edge node[auto,swap] {$\ppr'H$} (m-2-1)
   (m-1-2) edge node[auto] {$\Ph$} (m-2-2)
   (m-2-1) edge node[auto] {$\Theta$} (m-2-2)
  ;

%begin pullback symbol%
  \begin{scope}[shift=($(m-1-1)!.4!(m-2-2)$)]
  \draw +(-0.25,0) -- +(0,0)  -- +(0,0.25);
  \end{scope}
  %end pullback symbol%
\end{tikzpicture}
\end{equation}
is a pullback diagram in the category $\cAs$. Indeed, if $(B,d^B,Q^B) \xto{\Psi} (A,d,Q)$ and $(B,d^B,Q^B) \xto{\Psi'} (A',d',Q')$ are morphisms in $\cAs$ such that $\Phi\Psi=\Theta \Psi'$, then the $\infty$-morphism
\begin{equation} \label{eq:strict-pb1}
J \cc (\Psi \tim \Psi') \maps (B,d^B,Q^B) \to (\ti{A}, \ti{d},\ti{Q})
\end{equation}
is the unique morphism in $\cAs$ which makes the relevant diagrams commute. \hfill \qed

\vfill

\noindent\rule{4cm}{0.4pt}

{\footnotesize{
\noindent\textsc{Department of Mathematics and Statistics, University of Nevada,
 Reno. \\
1664 N. Virginia Street Reno, NV 89557-0084 USA}\\
\emph{E-mail address:} \textbf{pmilham@nevada.unr.edu, alexmilham.math@gmail.com}
\\

\noindent\textsc{Department of Mathematics and Statistics, University of Nevada,
 Reno. \\
1664 N. Virginia Street Reno, NV 89557-0084 USA}\\
\emph{E-mail address:} \textbf{chrisrogers@unr.edu, chris.rogers.math@gmail.com}

}
}

\end{document}